\documentclass[12pt, a4paper]{article}
\usepackage[a4paper,margin=2cm]{geometry}

\usepackage{soul}
\usepackage{url}
\usepackage{subeqnarray}
\usepackage{amsmath, amsfonts, amssymb, amsthm} 
\usepackage{algorithm,algorithmic}
\usepackage{graphicx}
\usepackage{authblk}
\usepackage{hyperref}
\usepackage{graphicx}
\usepackage{graphicx}
\usepackage{tikz}
\usepackage{pgfplots}
\usepackage{pgfplotstable}
\usepackage{subcaption}

\newcommand{\R}{\mathbb{R}}
\newcommand{\N}{\mathbb{N}}
\newcommand{\set}[1]{ \left\{ #1 \right\} }
\newcommand{\cH}{\mathcal{H}} 
\DeclareMathOperator{\diag}{diag}
\newcommand{\kernel}{k}
\newcommand{\fa}{\hbox{ for all }}
\DeclareMathOperator{\diam}{diam}

\newtheorem{theorem}{Theorem}[section]
\newtheorem{remark}{Remark}[section]

\usetikzlibrary{patterns,intersections}
\pgfplotsset{compat=1.14}
\usepgfplotslibrary{groupplots}

\title{Kernel methods for center manifold approximation and a data-based version of the Center Manifold Theorem}
\author[1]{B. Haasdonk\thanks{haasdonk@mathematik.uni-stuttgart.de}}
\author[2]{B. Hamzi\thanks{boumediene.hamzi@gmail.com, \href{https://orcid.org/0000-0002-9446-2614}{orcid.org/0000-0002-9446-2614}}}
\author[3]{G. Santin\thanks{gsantin@fbk.eu, \href{https://orcid.org/0000-0001-6959-1070}{orcid.org/0000-0001-6959-1070}}}
\author[1]{D. Wittwar\thanks{dominik.wittwar@mathematik.uni-stuttgart.de}}
\affil[1]{Institute for Applied Analysis and Numerical Simulation, University of Stuttgart, Germany}
\affil[2]{Department of Mathematics, Imperial College London, United Kingdom}
\affil[3]{Center for Information and Communication Technology, Bruno Kessler Foundation, Trento, Italy}

\begin{document}
\maketitle

\begin{abstract}
For dynamical systems with a non hyperbolic equilibrium, it is possible to significantly simplify the study of stability by means of the center manifold theory.
This theory allows to isolate the complicated asymptotic behavior of the system close to the equilibrium point and to obtain meaningful predictions of its behavior by analyzing a reduced order system on the so-called center manifold.

Since the center manifold is usually not known, good approximation methods are important as the center manifold theorem states that the stability properties of the origin of the reduced order system are the same as those of the origin of the full order system.

In this work, we establish a data-based version of the center manifold theorem that works by considering an approximation in place of an exact manifold. Also the error between the approximated and the original reduced dynamics are quantified. 

We then use an apposite data-based kernel method to construct a suitable approximation of the manifold close to the equilibrium, which is compatible with our general error theory. The data are collected by repeated numerical simulation of the full system by means of a high-accuracy solver, which generates sets of discrete trajectories that are then used as a training set. 
The method is tested on different examples which show promising performance and good accuracy. 

\end{abstract}

\section{Introduction}

Center manifold theory plays an important role in the study of the stability of dynamical systems when the equilibrium point is not hyperbolic.  It  isolates the complicated asymptotic behavior by locating the center manifold which is an invariant manifold tangent to the subspace spanned by the eigenspace of eigenvalues on the imaginary axis. Then, the dynamics of the original system  will be essentially determined by the  restriction of this dynamics on the center manifold since the local dynamic behavior ``transverse'' to this invariant manifold is relatively simple as it corresponds to the flows in the local stable (and unstable) manifolds. In practice, one does not compute the center manifold and its dynamics exactly since this requires the resolution of a quasilinear partial differential  equation which is not easily solvable. In most cases of interest, an approximation of degree two or three of the solution is sufficient. Then, the reduced dynamics on the center manifold can be determined, its stability can be studied and then conclusions about the stability of the original system can be obtained. For a detailed treatment of this topic, we refer e.g. to \cite{Carr81,henry,KELLEY1967336,kelley,pliss,sositaisvili}.

In this paper we consider data-based approximations of the center manifold, and our main result proves in great generality that any sufficiently accurate approximation can be used to draw conclusions on the asymptotic stability of the system, thus establishing one direction of a data-based center manifold theorem. 

We then turn our attention to concrete procedures to construct approximants that provide that kind of error control, and to this end we revise and refine a kernel-based approximation method that we introduced in \cite{Haasdonk2020a}. This approximation can be constructed based solely on the knowledge of points on the trajectories of the system close to the equilibrium point, that are obtained in practice by the numerical integration of the system. This method is formulated in the Reproducing Kernel Hilbert Spaces (RKHS) of a strictly positive kernel (see \cite{Aronszajn1950}), which have provided strong mathematical foundations for analyzing dynamical systems in other settings (see e.g. \cite{ALEXANDER2020132520,bh2020a,BH4,BHSIAM2017,BH6,Bruennette2019,Cavoretto2016a,lyap_bh,Haasdonk2020a,bhhokfs2020,bh2020b,klus2020data}).  
Our method explicitly incorporates into the approximant some prior knowledge on the structure of the manifold. This is possible using some RKHS theory, and we prove that this constrained problem admits a unique solution. 
In particular, we show how to control the error of this approximation, and the resulting bound guarantees that this kernel-based approximation is compatible with the general approximation discussed above. 
Moreover, these bounds are such that they can be applied even if the approximant is constructed using a subset of the dataset selected by means of certain greedy methods.

The paper is organized as follows. In Section \ref{sec:center_manifold} we recall the definition of the center manifold and its role in the study of the stability of certain equilibria. In Section \ref{sec:data_based_cm}, we prove a data-based version of the center manifold theorem. In section \ref{sec:kernel} we describe kernel methods to construct a data-based approximation of the center manifold. Some numerical experiments are given in Section \ref{sec: Numerical examples}.

\section{The center manifold}\label{sec:center_manifold}

We consider an open set $D\subset\R^n$ with $0\in D$, and a dynamical system
\begin{equation}\label{eqn:fors}
\dot{x}=f(x),
\end{equation}
where $f : D \rightarrow \R^n$ is a continuously differentiable function and $x:=x(t)\in D$ for all $t$. We are interested in the study of the behavior of this possibly high-dimensional 
system around an equilibrium point, which we may assume w.l.o.g. to be the origin, i.e., $f(0) = 0$.

Here and in the following of this paper, we use the notation $\nabla_x g(x):=\left(\partial_1 g(x), \dots, \partial_n g(x)\right)^T$ and $D_x f(x):=\left(\partial_i f_j(x)\right)_{1\leq i\leq n, 1\leq j\leq m}$ to denote the gradient of a function $g:\R^n\to\R$ and the 
Jacobian matrix of a vector-valued function $f:\R^n\to\R^m$. 
Moreover, a norm $\|\cdot\|$ 
without subscript refers to the Euclidean $2$-norm on $\R^n$ for some $n$, or to the induced norm of a matrix in $\R^{n\times n}$.

Letting $L:=D_x f(x)_{|{x=0}}\in \R^{n\times n}$, we can rewrite \eqref{eqn:fors} as
\begin{align*}
\dot{x}=f(x)=L x+N(x),
\end{align*}
with the nonlinear component $N(x):=f(x) - Lx$, and denote as $\sigma_{\R}(L)$ the set of real parts of the eigenvalues of $L$.

A classical result relates the 
stability of the equilibrium with the spectrum of $L$, and in particular it is known that if all the eigenvalues of $L$ have negative real part, i.e., 
$\sigma_{\R}(L)\subset \R_{<0}$, then the origin is asymptotically stable, and if $L$ has some eigenvalues with positive real parts, then the origin is 
unstable. If  
instead $\sigma_{\R}(L)\subset \R_{\leq0}$ then the linearization fails to determine the stability properties of the origin, and thus the analysis of this 
situation requires 
to employ additional tools.

In this case, we can first use a linear change of coordinates to separate the eigenvalues of $L$ with zero and with negative real part, i.e., we can 
rewrite \eqref{eqn:fors} as\footnote{Please note that $x$ will be a generic variable related to the state of a dynamical system
  and thus may have varying dimension and meaning throughout this paper.} 
\begin{align}\label{eq:full_order_system}
\dot{x}&= L_1 x + N_1(x,y),\\
\dot{y}&= L_2 y + N_2(x,y),\nonumber
\end{align}
where $L_1 \in \R^{d \times d}$ is such that $\sigma_{\R}(L_1)=\{0\}$ and $L_2 \in \R^{m \times m}$ with $m := n-d$ is such that 
$\sigma_{\R}(L_2)\subset\R_{<0}$.   The nonlinear functions $N_1 : \R^d \times \R^m \rightarrow \R^d$ and $N_2 : \R^d \times \R^m \rightarrow \R^m$ are 
continuously differentiable. Intuitively, we expect the stability of the equilibrium to only depend on the nonlinear term $N_1(x,y)$. This intuition turns out 
to 
be correct, and indeed it can be properly formalized by means of the Center Manifold Theorem. 

Before stating this theorem, we start by recalling a sufficient condition for the existence of a center manifold. Observe that the following assumption on the 
smoothness of $N_1, N_2$ require that 
$f\in C^2(D)$.
\begin{theorem}[Theorem 8.1 in \cite{khalil}]\label{th:carr}
Assume that $N_1$ and $N_2$ are twice continuously differentiable on $D$ and are such that
\begin{align}\label{conditions_N1_N2}
N_i(0,0)=0,\;\; D_x N_i(x,0)_{|x=0}=0,\;\; D_y N_i(0,y)_{|y=0}=0,\;\; i = 1, 2.
\end{align}
Furthermore, assume that the eigenvalues of $L_1$ have zero real part and the eigenvalues of $L_2$ have negative real part.

Then there exists a neighborhood $B 
\subset \R^d$ of the origin and a center manifold $h:B\to \R^m$ for  \eqref{eq:full_order_system}, i.e., 
$y=h(x)$ is an invariant manifold\footnote{A differentiable manifold $\cal{M}$ is said to be invariant under the flow of a vector 
field $X$ if for $x \in \cal{M}$, $F_t(x) \in \cal{M}$ for small $t > 0$, where $F_t(x)$ is the flow of $X$.} for \eqref{eq:full_order_system}, $h$ is continuously 
differentiable, and
\begin{align}\label{eq: center manifold properties}
h(0)=0,\;\; D_x h(0)=0.
\end{align}
\end{theorem}

Under the assumptions of this theorem, using \eqref{eq:full_order_system} we can deduce that $h $ satisfies the PDE
\begin{equation}\label{cm_pde} 
L_2 h(x) + N_2(x,h(x))=D_x h(x)\left(L_1 x +N_1(x,h(x))\right),
\end{equation}
and the following center manifold theorem ensures that there are continuously differentiable solutions to this PDE. Moreover, the theorem also allows to 
deduce the stability of 
the origin 
of 
the full 
order 
system \eqref{eq:full_order_system} from the stability of the origin of a reduced 
order system called the \emph{center dynamics}.

\begin{theorem}[Center Manifold Theorem (Theorem 8.2 and Corollary 8.2 in \cite{khalil})]\label{cm_theorem} 
Under the assumptions of Theorem \ref{th:carr}, the equilibrium $x=0, y=0$ of the full order system \eqref{eq:full_order_system} is asymptotically stable 
(resp. 
unstable) if and only if the equilibrium 
$x=0$ of the reduced dynamic 
(dynamics on the center manifold) 
\begin{equation}\label{reduced_system} 
\dot{x}= L_1 x +N_1(x,h(x)),
\end{equation}
is  asymptotically stable (resp. unstable).
\end{theorem}

This result guarantees that, upon solving the PDE \eqref{cm_pde} for the unknown $h:\R^d\to\R^m$, the problem of analyzing the stability properties of 
the system 
\eqref{eq:full_order_system}  reduces to analyzing the nonlinear stability of the lower dimensional system \eqref{reduced_system}. This second problem is of 
smaller 
dimension and thus, provided the knowledge of $h$, the approach is attractive to obtain information on the system \eqref{eqn:fors} via a reduced order model. 

Observe that the dimension $d$ of this reduced system is equal to the number of eigenvalues with zero real part (counted with multiplicity) of the original 
system 
\eqref{eqn:fors}, and thus it may take any value between $1$ and the full dimension $n$, depending on the particular dynamical system under 
consideration. The actual dimensional reduction provided by the center manifold is thus highly varying depending on the problem.

We remark that this methodology is valid also for parameterized dynamical systems and is used to study the stability of 
dynamical systems with bifurcations. 

\begin{remark}
An exact knowledge of the center manifold $h$ is not required for this equilibrium analysis, since it is sufficient to have an approximated
solution of the PDE 
\eqref{cm_pde}.
Indeed, it is frequently sufficient to compute only the low degree terms of the Taylor series expansion of $h$ around $x=0$, i.e.,  if $(\cdot)^{[k]}$ is 
the degree $k$ part of the Taylor series of $h$, the approximation
\begin{align}\label{eq:taylor}
h(x)\approx h^{[1]} x+h^{[2]}(x)+h^{[3]}(x)+\ldots+h^{[d-1]}(x)
\end{align}
is sufficient to obtain an approximation of the dynamics of order $\varepsilon^d$ as $\|x\|\leq \varepsilon$. 

The approximation \eqref{eq:taylor} can be 
obtained by 
coefficient comparison, thus rewriting the PDE \eqref{cm_pde} as a set of algebraic equations as 
\begin{align*}
&L_2 h^{[1]}=h^{[1]} L_1\\
&L_2 h^{[2]}(x) + N^{[2]}_2(x,h^{[1]}(x))=D_x h^{[2]}(x)\left(L_1 x_1 +N^{[2]}_1(x,h^{[1]}(x))\right)\\
&L_2 h^{[3]}(x) + \left(N_2(x,h^{[2]}(x))\right)^{[2]}=D_x h^{[2]}(x) \left(L_1 x +\left(N_1(x,h^{[2]}(x))\right)^{[2]}\right)\\
&\dots.
\end{align*}
Nevertheless, even this approximated knowledge of $h$ can be difficult to obtain in practice for a general ODE. 
\end{remark}

\section{A data-based version of the Center Manifold Theorem}\label{sec:data_based_cm}

In this section we show that one direction of Theorem \ref{cm_theorem} holds even when the center manifold is replaced by a sufficiently accurate 
approximation. Namely, we assume to have access only to an approximant $\hat h:\R^d\to\R^m$ of the center manifold, instead of $h$ itself, and we consider 
the approximated reduced 
dynamics
\begin{align*}
\dot{x}= L_1 x +N_1(x,\hat{h}(x)).
\end{align*}
This system is now both low dimensional and completely explicit, since $\hat h$ is known. Observe in particular that the change of coordinates required to 
obtain $L_1$ and $N_1$ from $L$ and $N$ is based solely on a diagonalization of $L$.

Depending on different conditions that the approximant may satisfy, we show in Theorem \ref{th:data_based_cm1} that the
asymptotic stability of the origin of the full order dynamics follows from the asymptotic stability of the approximated center 
dynamics. To formulate this data-based version of the Center Manifold Theorem, we first need to recall the following result. 
\begin{theorem}[Converse Lyapunov theorem (Theorem 4.16 in \cite{khalil})\footnote{The cited theorem actually provides more details on the bounds in 
\eqref{eq:conv_lyap1}, but to 
avoid introducing unnecessary complications we report here just the inequalities that will be used in the following.}]\label{th:conv_lyap}
Assume that $x=0$ is an asymptotically stable equilibrium of the system
\begin{align*}
\dot x = f(x),
\end{align*}
with $f:D\to \R^n$, $D\subset\R^n$ an open set such that $0\in D$, $f\in C^1(D)$, and $D_x f$ bounded in $D$.

Then there exists a neighborhood $B\subset D$ of zero and a Lyapunov function for the system, i.e., a continuously differentiable function $V:B\to \R$ such 
that $V(x)>0$ for all $x\in B\setminus 
\{0\}$, $V(0)=0$, and 
 \begin{align}
  \nabla_x V(x)^T f(x) &\leq -\alpha(\|x\|),\label{eq:conv_lyap1}\\
  \left\|\nabla_x V(x)\right\| &\leq\phantom{-} k\nonumber,  
 \end{align}
where $k>0$ is a constant depending on $B$ and $\alpha: [0, \infty) \rightarrow [0,\infty)$ is a strictly increasing continuous function with $\alpha(0)=0$.
\end{theorem}

We can now state and prove our main result. 
Actually, Theorem \ref{th:data_based_cm1}. can be understood in a more general way as a ``robust'' center manifold theorem as $\hat h$ does not necessarily need to be a data-based approximant. But as we only use such data-based approximants in the following, this terminology data-based Center Manifold Theorem seems appropriate.

\begin{theorem}[Data based Center Manifold Theorem - Part I]\label{th:data_based_cm1} 
Under the assumptions of Theorem \ref{th:carr}, assume that $D$ is bounded, let $B\subset D\subset\R^d$ be a neighborhood of zero, and let $\hat h:B\to\R^m$ be an approximant of the 
center 
manifold $h$. 
Define furthermore $\ell_D:=\sup_{(x, y)\in D}\left\|D_y N_1(x, y)\right\|<\infty$.

Assume that the point $x=0$ is an asymptotically stable equilibrium of the approximated center dynamics 
\begin{equation}\label{eq:red_appr_system} 
\dot{x}= L_1 x +N_1(x,\hat{h}(x)).
\end{equation}
Let $V$ be the function obtained by applying Theorem \ref{th:conv_lyap} to \eqref{eq:red_appr_system}, let $\alpha(\cdot)$ and $k$ be as in 
\eqref{eq:conv_lyap1}, and assume that there exists $p>0$ such that 
\begin{align}\label{eq:alpha_asympt}
0<\lim\limits_{\|x\|\to 0} \frac{\alpha(\|x\|)}{\phantom{\alpha(}\|x\|^p}<\infty. 
\end{align}
Under these assumptions, define
\begin{align}\label{eq:epsilon_zero}
\varepsilon_0:= \frac{1}{k \ell_D}\cdot \inf\limits_{x\in B\setminus\{0\}}  \frac{\alpha(\|x\|)}{\phantom{\alpha(}\|x\|^p}>0.
\end{align}
Then, if there exist $q\geq p$ and $\varepsilon\leq\varepsilon_0$ such that 
\begin{align}\label{eq:error_hypothesis1}
\left\|h(x) - \hat h(x)\right\| \leq \varepsilon \|x\|^q \quad \fa x \in B,
\end{align}  
the equilibrium $x=0, y=0$ of the full order dynamics \eqref{eq:full_order_system} is asymptotically stable.
\end{theorem}

\begin{proof}
First we prove that if $x=0$ is an equilibrium of \eqref{eq:red_appr_system} then $(x, y) = (0, 0)$ is an equilibrium of the full order system. Indeed, since 
$h(0)=0$ by definition, we have that \eqref{eq:error_hypothesis1} implies that $\hat h(0) = 0$ whatever is the value of $q>0$. Thus, since $x=0$ is an 
equilibrium of \eqref{eq:red_appr_system}, we have
\begin{align*}
0=L_1 0 +N_1(0,\hat{h}(0)) = L_1 0 +N_1(0,0)= L_1 0 +N_1(0,{h}(0)),
\end{align*}
and thus $x=0$ is an equilibrium of the reduced system \eqref{reduced_system}. By Theorem \ref{cm_theorem} it follows that $(0, 0)$ is an equilibrium of 
\eqref{eq:full_order_system}.

In order to prove asymptotic stability we adapt the proof of Theorem 8.2 in \cite{khalil}. We consider 
the change of variables
\begin{align*}
\left[\begin{array}{c} x \\ w \end{array} \right]= \left[\begin{array}{c} x \\ y-h(x) \end{array} \right],
\end{align*}
which transforms \eqref{eq:full_order_system} into
\begin{align}\label{eq:full_new_vars} 
 \dot{x} &= L_1 x + N_1(x,w+h(x))\nonumber\\
 \dot{w} &= L_2(w+h(x))+N_2(x,w+h(x))-D_x h(x) \left(L_1 x+N_1(x,w+h(x))\right),
\end{align} 
and this transformation is such that this new system has the same stability properties of \eqref{eq:full_order_system} (see \cite{khalil}).

Observe that the difference $y(t)-h(x(t))$ represents the deviation of the trajectory from the center manifold  at any time $t$. This may be identically 
zero if $y(0)=h(x(0))$ since in this case the solution $(x(t),y(t))$ will lie in the manifold for all $t \ge 0$, and if so the second equation in 
\eqref{eq:full_new_vars} is trivially true.

Defining now
\begin{align*}
{\cal N}_1(x,w)&:=N_1(x,w+h(x))-N_1(x,{h}(x)),\\
{\cal N}_2(x,w)&:=N_2(x,w+h(x)) -N_2(x, h(x)) -D_x h(x)\ {\cal N}_1(x,w),
\end{align*}
it is proven in \cite{khalil} that by manipulation of the previous equations it follows that 
\begin{align}
 \dot{x}&= L_1 x + N_1(x,{h}(x))+ {\cal N}_1(x,w),\\
 \dot{w} &= L_2 w + {\cal N}_2(x,w)\nonumber,
\end{align} 
and furthermore that for any radius $\rho>0$ there exist constants $k_i:=k_i(\rho)>0$ such that for all $(x, w)$ in the ball $B_\rho:=B_\rho(0, 0)$ with 
center $(0, 
0)$ and radius $\rho$ it holds
\begin{align}\label{eq:bound_on_N_i}
\left\|{\cal N}_i(x,w)\right\| \leq k_i \left\|w\right\|,\quad i=1,2, 
\end{align}
where the constants $k_1$ and $k_2$ can be made arbitrarily small by choosing $\rho$ small enough.

We now consider the reduced order system \eqref{eq:red_appr_system}
and show that if the origin $x=0$ is asymptotically stable, then the origin $(x,w)=(0,0)$ is asymptotically stable for the full system in the form 
\eqref{eq:full_new_vars}.

Since $x=0$ is asymptotically stable, by Theorem \ref{th:conv_lyap} there are a continuously differentiable and 
positive function $V$ and a constant $k>0$ such that in a neighborhood $B$ of the origin $x=0$ it holds
\begin{align}\label{eq:bound_on_de_v}
\nabla_x V(x)^T \left(L_1 x +N_1(x,\hat{h}(x))\right)& \le -\alpha(\|x\|),\\
\nonumber\left\| \nabla_x V(x)\right\|\ &\le\phantom{-}  k,
\end{align}
where $\alpha: [0, \infty) \rightarrow [0,\infty)$ is a strictly increasing continuous function with $\alpha(0)=0$. Furthermore, by choosing $\rho$ small 
enough we may assume that $x\in B$ for all $(x,w)\in B_\rho$, and thus both the bounds \eqref{eq:bound_on_N_i} and \eqref{eq:bound_on_de_v} hold in $B_{\rho}$.

Moreover, in the proof of Theorem 8.2 in \cite{khalil} it is shown that there is a matrix $P$ which is positive definite and solves the Lyapunov equation 
$L_2^TP+PL_2=-I$, where $I$ is the identity matrix. We denote as $\lambda_{\min}, \lambda_{\max} >0$ the minimal and maximal eigenvalues of this matrix. 

Using $V$ and $P$, we define a candidate Lyapunov function
\begin{equation}
{\cal V}(x,w):=V(x)+\sqrt{w^T P w}  
\end{equation}
for the full order system. 
By construction this function is zero in $(x,w)=(0,0)$ and strictly positive in $B_\rho\setminus\{(0,0)\}$. If we show that its derivative along 
the trajectories of the full order system is strictly negative, it follows (see Theorem 4.1 in \cite{khalil}) that 
the origin $(x, w) = (0, 0)$ is asymptotically stable for the full system.

To see this, we compute the derivative of ${\cal V}$ along the trajectories of the full order system as
\begin{align}\label{eq:lyapunov}
\dot{{\cal V}}(x,w)
= \nabla_x &V(x)^T \left(L_1 x + N_1(x,h(x))+{\cal N}_1(x,w)\right)+\nonumber\\
&+\frac{1}{2 \sqrt{w^TPw}}\left(w^T\left(PL_2+L_2^TP\right)w+2w^TP {\cal N}_2(x,w)\right).
\end{align}
We take from the proof of Theorem 8.2 in \cite{khalil} the estimate of the second term, i.e., 
\begin{align}\label{eq:first_term_bound}
\frac{1}{2 \sqrt{w^TPw}}\left(w^T\left(PL_2+L_2^TP\right)w+2w^TP {\cal N}_2(x,w)\right) 
&\leq -\left(\frac{1}{2 \sqrt{\lambda_{\max}}}-\frac{k_2 \lambda_{\max}} {\sqrt{\lambda_{\min}}}\right) \|w\|.
\end{align}
To bound instead the first term of \eqref{eq:lyapunov}, we add and subtract $N_1(x,\hat h(x))$, we use the two inequalities \eqref{eq:bound_on_de_v} and the 
bound \eqref{eq:bound_on_N_i} 
with $i=1$, and via the Cauchy-Schwartz inequality we obtain
\begin{align}\label{eq:second_term_bound}
\partial_x &V(x)^T\left(L_1 x + N_1(x,h(x))+{\cal N}_1(x,w)\right)=\nonumber\\
&=\nabla_x V(x)^T\left(L_1 x + N_1(x,\hat h(x)) - N_1(x,\hat h(x)) + N_1(x,h(x))+{\cal N}_1(x,w)\right)\nonumber\\
&=\nabla_x V(x)^T\left(L_1 x + N_1(x,\hat h(x))\right) + \nabla_x V(x)^T\left(N_1(x,h(x)) - N_1(x,\hat h(x))+ {\cal N}_1(x,w)\right)\nonumber\\
&\leq -\alpha\left(\|x\|\right) + \left\|\nabla_x V(x)\right\| \left\|N_1(x,h(x)) - N_1(x,\hat h(x))+ {\cal N}_1(x,w)\right\|\nonumber\\
&\leq -\alpha\left(\|x\|\right) + k \left\|N_1(x,h(x)) - N_1(x,\hat h(x))\right\| + k \left\|{\cal N}_1(x,w)\right\|\nonumber\\
&\leq -\alpha\left(\|x\|\right) + k \left\|N_1(x,h(x)) - N_1(x,\hat h(x))\right\|+ k k_1 \|w\|.
\end{align}
Moreover, since $N_1$ is continuously differentiable by assumption, for all $x$ the function $w\mapsto N_1(x, w)$ is Lipschitz 
continuous for $(x, w)\in  B_\rho$, with a Lipschitz constant that is bounded by $\ell_{\rho}:=\sup_{(x, y)\in B_\rho}\left\|D_y N_1(x, 
y)\right\|<\infty$. In particular since $B_\rho\subset D$ it holds  $\ell_{\rho}\leq \ell_D$ and 
\begin{align}\label{eq:second_term_bound_2}
\left\|N_1(x, h(x)) - N_1(x, \hat h(x))\right\|
&\leq \ell_{\rho}\left\|h(x) - \hat h(x)\right\|
\leq \ell_D\left\|h(x) - \hat h(x)\right\|
\quad \fa x\in B_\rho.
\end{align}
Inserting \eqref{eq:first_term_bound}, \eqref{eq:second_term_bound} and \eqref{eq:second_term_bound_2} into \eqref{eq:lyapunov} gives for all $(x, w) \in 
B_\rho$ the bound
\begin{align*}
\dot{{\cal V}}(x,w)
&\leq -\alpha\left(\|x\|\right) +  k  \ell_D\left\|h(x) - \hat h(x)\right\|
-\left(\frac{1}{2 \sqrt{\lambda_{\max}}}-\frac{k_2 \lambda_{\max}} {\sqrt{\lambda_{\min}}}-k k_1\right) \|w\|\\
&= -\alpha\left(\|x\|\right) +  k  \ell_D\left\|h(x) - \hat h(x)\right\| - \frac{1}{4 \sqrt{\lambda_{\max}}}\|w\|
-\left(\frac{1}{4 \sqrt{\lambda_{\max}}}-\frac{k_2 \lambda_{\max}} {\sqrt{\lambda_{\min}}}-k k_1\right) \|w\|.
\end{align*}
The term $\left(\frac{1}{4 \sqrt{\lambda_{\max}}}-\frac{k_2 \lambda_{\max}} {\sqrt{\lambda_{\min}}}-k k_1\right)$ is always strictly positive if 
$k_1$, $k_2$ are made sufficiently small by choosing an appropriately small value of $\rho$. Thus, we are left with the 
inequality
\begin{align*}
\dot{{\cal V}}(x,w)
&\leq -\alpha\left(\|x\|\right) +  k  \ell_D\left\|h(x) - \hat h(x)\right\| - \frac{1}{4 \sqrt{\lambda_{\max}}}\|w\|,\;\; (x, w) \in B_\rho,
\end{align*}
and using the error bound \eqref{eq:error_hypothesis1} we get that
\begin{align}\label{eq:final_bound}
\dot{{\cal V}}(x,w) \le k \ell_D \varepsilon \|x\|^q -\alpha(\|x\|)-\frac{1}{4 \sqrt{\lambda_{\max}}} \|w\|,\;\; (x, w) \in B_\rho.
\end{align}
If $x=0$ and $w\neq 0$, the last term ensures that $\dot{{\cal V}}(x,w)<0$.
Moreover, the term  $k \ell_D \varepsilon \|x\|^p -\alpha(\|x\|)$ is strictly negative in $B_\rho\setminus\{x=0\}$ exactly when 
\begin{align*}
 \varepsilon  <\frac{1}{k \ell_D}\frac{\alpha(\|x\|)}{\phantom{\alpha(}\|x\|^q}\;\; \fa (x, 0) \in B_{\rho}\setminus\{x=0\},
\end{align*}
i.e., when $\varepsilon\leq \varepsilon_0$, which is true by assumption. 

It follows that $\dot{{\cal V}}(x,w)<0$ for all $(x, w) \in B_\rho\setminus \{(0, 0)\}$, and thus the 
origin $(0, 0)$ is asymptotically stable for the full order system \eqref{eq:full_new_vars}, and we are done.
\end{proof}

We conclude this section with some remarks on the theorem and on its proof.

\begin{remark}[Weaker error conditions - Error rate]
Observe that the error statement \eqref{eq:error_hypothesis1} is sharp in $q$, since the assumption that the approximation error behaves as $\alpha$ when 
$\|x\|\to0$, i.e., that \eqref{eq:error_hypothesis1} holds with $q\geq p$, is crucial for the proof. 
Indeed, if only an error of order $\|x\|^q$ with $q<p$ can be obtained in 
\eqref{eq:error_hypothesis1}, then \eqref{eq:epsilon_zero} would give $\varepsilon_0=0$. Thus it would 
not be possible to conclude the proof by showing that $\dot{{\cal V}}(x,w)<0$ in a neighborhood of $(0, 0)$, no matter how small we choose the neighborhood, 
unless $\hat h = h$.

This is the case also if one is trying to prove stability instead of asymptotic stability of the origin of the full order system. Indeed, in this case it 
would be sufficient to prove that $\dot{{\cal V}}(x,w) \le 0$ in $B_\rho\setminus \{0\}$ (i.e., equality is possible), but also in this case the same problem 
as before would persist. 
\end{remark}

\begin{remark}[Weaker error conditions - Error constant]
The error bound \eqref{eq:error_hypothesis1} is instead not necessarily optimal with respect to the constant $\varepsilon$. Indeed, in \eqref{eq:epsilon_zero} 
we defined for simplicity $\varepsilon_0$ as a function of $\ell_D$, which is the Lipschitz constant of $N_2$ in $D$ with respect to its second variable. 
Inspecting the proof, it is clear that it would instead be sufficient to define $\varepsilon_0$ in terms of $\ell_{\rho}$ (see \eqref{eq:second_term_bound_2}). 
Since $\ell_{\rho}\leq \ell_D$, the new value of $\varepsilon_0$ would be larger, and thus the condition $\varepsilon\leq \varepsilon_0$ could be made less 
restrictive.
\end{remark}

\begin{remark}[Unstable equilibrium]
The Center Manifold Theorem \ref{cm_theorem} guarantees also that if the origin is unstable in the reduced system, then $(0, 0)$ is unstable in the full 
system. We do not see how to extend this result to our case.
\end{remark}

\begin{remark}[Data based Center Manifold Theorem - Part II]
We currently have no proof of the other implication of the Center Manifold Theorem, i.e., proving that if the origin of the full order dynamics is 
asymptotically stable, then also the origin of the approximated reduced dynamics is asymptotically stable.

This converse implication would be of practical interest in order to guarantee the well-posedness of the problem of studying the asymptotic stability of an 
equilibrium of the reduced and approximated system. 
Indeed, this result would ensure that if one starts from a full system with an asymptotically stable equilibrium 
and constructs a sufficiently accurate approximant, then also this approximant has an asymptotically stable equilibrium.
\end{remark}

\begin{remark}[Finding the correct value of $q$]
We will see in the following sections that it is possible to obtain data-based approximants $\hat h$ which satisfy the error bound 
\eqref{eq:error_hypothesis1} with arbitrarily large $q$. Nevertheless, this process requires that $q$ is known before constructing $\hat h$, but 
in order to apply Theorem \ref{th:data_based_cm1} we need $q\geq p$, and the value of $p$ depends on the behavior of the Lyapunov function $V$, 
which in turns depends on $\hat h$ itself. This circular argument may rule out the possibility to design a systematic way to find the 
correct value of $q$.

To solve this issue, it would be interesting to know if the exponent $p$ in \eqref{eq:alpha_asympt} can be deduced from properties of the full system alone, 
independently of $\hat h$. Namely, one may expect that if the exact reduced system has a Lyapunov function $V_h$ which satisfies \eqref{eq:alpha_asympt} with a 
given $p$, then any sufficiently accurate approximant $\hat h$ would define an approximated reduced system \eqref{eq:red_appr_system} with a Lyapunov 
function $V_{\hat h}$ which satisfies \eqref{eq:alpha_asympt} with the same $p$. 

This kind of results may be possibly derived from results on the perturbation of Lyapunov functions (see e.g. \cite{Lin1996}).
\end{remark}

\subsection{Evolution of the system along the approximated reduced dynamics}

An approximated center manifold can also be used to define approximated trajectories of the reduced system. In the following statement we show that 
the approximation error of $\hat h$ can be used to bound the deviation between the true and the reduced center dynamics.

In the following we denote as $\diam(B)$ the diameter of the set $B$.

\begin{theorem}\label{th:perturbation}
Under the assumptions of Theorem \ref{th:carr}, let $B\subset\R^d$ be a neighborhood of zero, and let $\hat h:B\to\R^m$ be an approximant of the 
center manifold $h$ such that there exist $q\geq 0$ and $\varepsilon>0$ with
\begin{align}\label{eq:error_hypothesis2}
\left\|h(x) - \hat h(x)\right\| \leq \varepsilon \|x\|^q \quad \fa x \in B.
\end{align}  
Furthermore, define the Lipschitz constants
\begin{align*}
\ell_1:=\sup_{x\in B}\left\|D_x\left(L_1 x + N_1(x, h(x))\right)\right\|,\quad \ell_2  &:= \sup_{x\in B}\left\|D_y N_1(x, y)_{|y=h(x)}\right\|.
\end{align*}
If $x_0\in B$, and if $x(t), z(t)$ are the solutions of the initial value problems 
\begin{align*}
\left\{\begin{array}{llcll}
\dot{x}&= L_1 x +N_1(x,{h}(x))\\
x(0) &= x_0
\end{array}\right.
,\quad
&
\quad
\left\{\begin{array}{ll}
\dot{z}&= L_1 z +N_1(z,\hat{h}(z))\\
z(0) &= x_0
\end{array}\right.,
\end{align*}
then for all $t>0$ it holds
\begin{align}\label{eq:perturbed_evolution}
\left\|x(t) - z(t)\right\| \leq \varepsilon\frac{\ell_2}{\ell_1}\diam(B)^q \left(e^{\ell_1\cdot t} - 1\right).
\end{align}
\end{theorem}

\begin{proof}
We define $g:\R^n \to \R^n$ as
\begin{align*}
g(x):= \left(L_1 x +N_1(x,{h}(x))\right) - \left(L_1 x +N_1(x,\hat{h}(x))\right)
=N_1(x,{h}(x)) - N_1(x,\hat{h}(x)),
\end{align*}
and set $\mu_B:= \sup_{x\in B}\left\|g(x)\right\|$. 
Then a classical perturbation result (see e.g. Theorem 3.4 in \cite{khalil}) states that
\begin{align*}
\left\|x(t) - z(t)\right\| \leq \frac{\mu_B}{\ell_1} \left(e^{\ell_1\cdot t} - 1\right)\quad \fa t>0.
\end{align*}
Following the same steps as in the proof of Theorem \ref{th:data_based_cm1} and using the bound \eqref{eq:error_hypothesis2} we derive the bound
\begin{align*}
\mu_B
&\leq \ell_2 \sup\limits_{x\in B}\left\|h(x) - \hat h(x)\right\| 
\leq \ell_2 \varepsilon \sup\limits_{x\in B}\|x\|^q \leq \ell_2 \varepsilon \diam(B)^q,
\end{align*}
and this concludes the proof.
\end{proof}

\section{Kernel-based approximation of the center manifold}\label{sec:kernel}
We discuss now a concrete approach for the construction of a data-based approximation of the center manifold, based on kernel methods.
The algorithm that we are discussing has been introduced in \cite{Haasdonk2020a}. We recall it here for completeness, and in addition we present a proof of the optimality of the approximant and we derive error bounds that may be combined with Theorem \ref{th:data_based_cm1} and Theorem \ref{th:perturbation}.

The approximant $\hat h$ to the center manifold is constructed in a data-based way, i.e., we assume to know the values taken by the true manifold $h$ on a finite 
set of input values. To collect these input-output pairs one can first derive the splitting \eqref{eq:full_order_system} by analyzing the eigenvalues and eigenvectors of the linearization $L:=D_x f(x)_{|{x=0}}\in \R^{n\times n}$, as discussed in Section \ref{sec:center_manifold}. When this splitting is known, it is possible to consider several initial conditions $\left(x_0, y_0\right)$ and the corresponding Initial Value Problems (IVPs) given by the coupling of the initial values with the dynamics \eqref{eq:full_order_system}. These IVPs may be solved numerically with high accuracy to collect sequences of pairs $(x_i, y_i)$ that are numerical trajectories of the IVP at a given time instant, and which are collected into a dataset $X_N:=\{x_i\}_{i=1}^N$, $Y_N:=\{x_i\}_{i=1}^N$ for a certain $N\in\N$. Thanks to Theorem \ref{th:carr}, we know that if the point $(x_i, y_i)$ is sufficiently close to $(0, 0)$, then $y_i\approx h(x_i)$. These sets $X_n$, $Y_n$ can then be regarded as collections of sufficiently accurate pairs of input-output values of the true manifold $h$.

Observe that a center manifold $h$ needs not to be unique, as in general multiple solutions of the PDE \eqref{cm_pde} may exist. Nevertheless, the solutions may differ at most by a term exponentially decaying to zero, which may thus be neglected (in terms of approximation) if the points $(x_i, y_i)$ are sufficiently close to the origin. To be more precise, it may be that the points in $X_N$, $Y_N$ of larger magnitude lie on different manifolds, but the distance between these different branches becomes negligible compared to the approximation error, which is assumed to be algebraic (see Theorem \ref{th:data_based_cm1}).

\subsection{Definition, existence and optimality of the approximant}
For the construction of the approximant we employ here matrix-valued kernels, which are suitable to approximate vector-valued functions.
Details on kernel-based approximation can be found e.g.~in~\cite{Wendland2005}, and the extension to the vectorial case is detailed e.g.~in~\cite{Micchelli2005, Wittwar2018}. Here we recall only that a positive definite matrix-valued kernel on a set $\Omega$ is a function $\kernel : \Omega \times \Omega \rightarrow \R^{m \times m}$ such that 
\begin{itemize}
\item $\kernel(x, y) = \kernel(y,x)^T$ for all $x,y \in \Omega$, 
\item the matrix $\left(\kernel(x_i, x_j)\right)_{i,j=1}^N\in\R^{mN\times mN}$ is positive semidefinite for any set $\{x_1, \dots, x_N\}\subset\Omega$ of pairwise distinct points, for all $ N \in\N$.
\end{itemize}
Associated to a positive definite kernel there is a unique Hilbert space $\mathcal H$ of vector valued functions $\Omega\to\R^m$, named native space, where the kernel is reproducing, meaning that $\kernel(\cdot, x) \alpha$ is the Riesz 
representer of the directional point evaluation $\delta_x^{\alpha}(f):= \alpha^T f(x)$, for all $\alpha\in\R^m$, $x\in\Omega$. This is the functional setting where the following approximation process is defined.

In our case we aim at an $\hat h$ that provides an error bound like \eqref{eq:error_hypothesis1}, and to achieve this we consider a kernel-based approximant that includes derivative information. For now we concentrate on a case that will provide an error bound of order $q=2$ since it is somehow more natural given the definition of the center manifold, and we will discuss the extension to general $q$ in the following.

To make this formal, we consider a twice continuously differentiable matrix-valued kernel $\kernel$ on $\Omega$ and define the approximant as the minimizer of a cost functional on $\mathcal H$, which is defined according to the following three requirements. 

First, the conditions in \eqref{eq: center manifold properties} i.e., $\hat h(0) = 0, D \hat h(0) = 0$, should be exactly met, since they are known properties of the center manifold.

Second, the approximant should provide a small distance between the measured output $y_i$ and the prediction $\hat h(x_i)$, for all $(x_i, y_i)$ in the dataset.
To this end we introduce for each index $i$ a positive definite weight matrix $\omega_i \in \R^{m \times m}$ which associates a possibly different importance to different data points and to different output dimensions. 
Using these matrices, we define a term $\sum\limits_{i=1}^N (\hat h(x_i) - y_i)^T \omega_i (\hat h(x_i) - y_i)$ in the cost functional that measures the deviation of the predictions of the approximant from the data.

Third, as discussed at the beginning of this section we do not know if a data pair $(x_i,y_i)$ lies exactly on the center manifold, i.e., if $y_i = h(x_i)$ holds, and thus an approximant which merely interpolates the data seems ill-suited for our purposes. For this reason, we introduce a regularization term in the functional to allow non-interpolatory solutions, as long as they have a small norm.

All together, we compute our approximant by minimizing over $s\in \cH$ the functional $J : \cH \rightarrow \R$ defined by
\begin{equation}
 J(s) := \| s \|_{\cH}^2 + \sum\limits_{i=1}^N (s(x_i) - y_i)^T \omega_i (s(x_i) - y_i), \label{eq: Functional}
\end{equation}
under the constraint $s(0) = 0, Ds(0) = 0$.

\begin{remark}
The weight matrices $\omega_i$ can either be chosen manually, or a regularizing function $r : \Omega \rightarrow \R^{m \times m}$ can be prescribed
such that $\omega_i := r(x_i)$ is symmetric and positive definite. In our numerical examples in Section \ref{sec: Numerical examples} we chose a constant 
regularization 
function, i.e.
\begin{equation}\label{rem:constant_regularization}
 \omega_i = r(x_i) = \lambda I_m
\end{equation}
for some $\lambda > 0$ independent of $i$. However, one might consider a more general approach, where the weight increases as the data tends to the origin,
i.e. $\omega_i \succeq \omega_j$ if $ \|x_i\| \leq \|x_j\|$, to give more importance to the samples that are expected to represent more closely the center manifold.
\end{remark}

We now show that the problem that we just defined via \eqref{eq: Functional} has a unique minimizer $\hat h$. The proof is taken from \cite{Wittwar2020thesis}, and we report it here for the sake of completeness.

\begin{theorem}[Representer Theorem for the approximation of the center manifold]\label{th:representer}
The minimization problem
\begin{align*}
\begin{array}{lll}
\min\limits_{s \in \cH} & J(s) &= \| s \|_{\cH}^2 + \sum\limits_{i=1}^N (s(x_i) - y_i)^T \omega_i (s(x_i) - y_i)\\
\text{ s.t. } &s(0) &= 0\\
& Ds(0) &= 0
\end{array}
\end{align*}
has a unique solution $\hat h \in \cH$ which can be represented as
\begin{align}
 \hat h(x) & = \sum\limits_{i=1}^{N+1} \kernel(x,x_i)\alpha_i + \sum\limits_{i=1}^d \partial_i^{(2)}\kernel(x,0)\beta_{i}, \label{eq: representer thm}
\end{align}
where $x_{N+1} := 0$, and where the superscript $\partial^{(2)}$ denotes that the derivative with regards to the second kernel component is taken.
Moreover, the coefficient vectors $\alpha_i, \beta_{i} \in \R^m$ can be computed by solving the system
\begin{equation}\label{eq:thelinearsystem}
 \begin{pmatrix}
  A + W & B \\
  B^T & C 
 \end{pmatrix}
 \begin{pmatrix}
  \boldsymbol{\alpha} \\
  \boldsymbol{\beta}
 \end{pmatrix}
  =
 \begin{pmatrix}
  \boldsymbol{Y} \\
  \boldsymbol{Z}
 \end{pmatrix},
\end{equation}
with
\begin{align*}
 A & := \left( \kernel(x_i,x_j) \right)_{i,j} \in \R^{m(N+1) \times m(N+1)}, \\
 W & := \diag\left( \omega_1^{-1},\dots, \omega_N^{-1},0 \right) \in \R^{m(N+1) \times m(N+1)}, \\
 B & := \left( \partial_{j}^{(2)}\kernel(x_i,0) \right)_{i,j} \in \R^{m(N+1) \times d m}, \\
 C & := \left( \partial_{i}^{(1)}\partial_{j}^{(2)}k(0,0) \right)_{i,j} \in \R^{d m \times d m}, \\
 \boldsymbol{Y} & := (y_1^T, \dots,y_n^T,0)^T \in \R^{m(N+1)}, \\
 \boldsymbol{Z} & := 0 \in \R^{d m \times m}.
 \end{align*}
\end{theorem}

\begin{proof}
 Let $\hat h \in \cH$ be given by \eqref{eq: representer thm}. Then it follows that $\hat h(X) = \boldsymbol{Y} - W \boldsymbol{\alpha}$, 
$\hat h(0) = 0$ and $D\hat h(0) = 0$, i.e. $\hat h$ lies in the admissible set for the minimization problem. Let $s \in \cH$ be an alternative admissible element. Then 
there exists $g \in \cH$ with $g(0) = 0$, $Dg(0) = 0$ such that $s = \hat h + g$. It now holds
 \begin{align*}
  J(s) & = J(\hat h+g) = (\hat h(X) + g(X) - \boldsymbol{Y})^T W (\hat h(X) + g(X) - \boldsymbol{Y})^T + \| \hat h + g \|_{\cH}^2 \\
  & = (\hat h(X) - \boldsymbol{Y})^TW^{+}(\hat h(X) - \boldsymbol{Y}) + \|\hat h\|_{\cH}^2 \\
  & \quad + 2(\hat h(X) - \boldsymbol{Y})^T W^{+}g(X) + g(X)^TW^{+}g(X) + \|g\|_{\cH}^2 + 2\langle \hat h,g\rangle \\
  & = J(\hat h) + \|g\|_{\cH}^2 + g(X)^TW^{+}g(X) + 2( (\hat h(X) - \boldsymbol{Y})^Tg(X) + \langle \hat h,g\rangle).
 \end{align*}
 Hence, it is sufficient to show that
 \begin{align*}
  (\hat h(X) - \boldsymbol{Y})^TW^{+}g(X) = - \langle \hat h,g \rangle.
 \end{align*}
 Since $\kernel$ is the reproducing kernel of $\cH$, and since $Dg(0) = 0$, it holds
 \begin{align*}
  \langle \hat h,g\rangle = \left\langle \sum\limits_{i=1}^{N+1} \kernel(\cdot,x_i)\alpha_i + \sum\limits_{i=1}^{m} \partial_i^{(2)} \kernel(\cdot,0)\beta_i, g 
\right\rangle = g(X)^T \boldsymbol{\alpha}.
 \end{align*}
 On the other hand we have
 \begin{align*}
  (\hat h(X) - \boldsymbol{Y})^TW^{+}g(X) = (-W\boldsymbol{\alpha})^TW^{+}g(X) = \alpha^T WW^{+}g(X) = \alpha^Tg(X).
 \end{align*}
Therefore,
\begin{align*}
 J(s) = J(\hat h) + \|g\|_{\cH}^2 + g(X)^TW^{+}g(X) \geq J(\hat h)
\end{align*}
and equality holds if and only if $g = 0$. Hence, $\hat h$ is the unique minimizer of $J$.
\end{proof}

Observe in particular that the expression \eqref{eq: representer thm} allows easy manipulation of the approximant $\hat h$. For example, first order derivatives are simply given by the formula
\begin{align}
\nonumber D\hat h(x) & =  \sum\limits_{i=1}^{N+1} D^{(1)}\kernel(x,x_i)\alpha_i + \sum\limits_{i=1}^m D^{(1)}\partial_i^{(2)}\kernel(x,0)\beta_{i}.
\end{align}

\begin{remark}[Higher order derivatives]\label{rem:higher_order}
Similarly to the case defined in Theorem \ref{th:representer}, it is possible to define a modified optimization problem where the same cost functional $J$ from \eqref{eq: Functional} is used, but additionally also constraints on higher order derivatives are set. More specifically, for any $p>1$, $p\in\N$, one may consider the set of multi-indexes of the form $a:=(a_1, \dots, a_d)\in \N^d$ with $|a|:= \sum_{i=1}^d a_i$, and require that 
\begin{align}\label{eq:high_ord_cond}
\partial^{a} \hat h_i(0) = \partial^a h_i(0), \;\;1\leq i\leq m,\;\; |a|\leq p,
\end{align}
where $\partial^a f(x) :=\partial_{1}^{a_1}\cdot\dots \partial_{d}^{a_d} f(x)$ for $f:\R^d\to\R$.

Assuming that the kernel and the manifold $h$ are smooth enough to impose these conditions, the optimization problem still has a unique solution of the form \eqref{eq: representer thm}, but where now the second sum is a linear combination of all the $p$-th order derivatives of $K$. Moreover, the coefficients can still be found by solving a linear system like \eqref{eq:thelinearsystem}, augmented with blocks corresponding to higher order derivatives.

The main limitation of this approach is that imposing the conditions \eqref{eq:high_ord_cond} requires the knowledge of $\partial^a h_i(0)$, $1\leq i\leq m$, for all $|a| \leq p$ and $p>1$. These values are not known in general, unlike in the case $p=1$ where they are defining properties of the center manifold.
\end{remark}

\subsection{Selection of the sampling points}\label{sec:greedy}
If the technique of the previous section is used as it is, the approximant \eqref{eq: representer thm} is given by an expansion with $N+1 + m$ terms, where 
${N}$ is the number of points in the training set $X_N, Y_N$. Therefore, it might be inefficient to evaluate a model that is built using a large dataset. Furthermore, the 
computation of the coefficients in \eqref{eq: representer thm} requires the solution of the linear system \eqref{eq:thelinearsystem} with $m (N+1) + m^2$ rows and columns, and which can be severely ill-conditioned for non well-placed points $X_N$.

To mitigate both problems, we employ an algorithm that aims at selecting small subsets $X_{\hat N}, Y_{\hat N}$ of $X_N$, $Y_N$ such that the approximant computed with 
these sets is a sufficiently good approximation of the one which uses the full sets. The algorithm selects the points in a greedy way, i.e., one point at a 
time 
is 
selected and added to the current training set. In this way, it is possible to identify a good set without the need to solve a nearly infeasible combinatorial 
problem. 

The selection is performed using the $P$-greedy method (see \cite{DeMarchi2005}, and \cite{Wittwar2019x} for the matrix-valued case) applied to the kernel $K$, such that the set of points is selected before the 
computation 
of the approximant.
The number of points $\hat N$, and therefore the expansion size and evaluation speed,
is depending on a prescribed target accuracy $\varepsilon_{tol}>0$. 
For details on the method implementation and its convergence properties we refer to \cite{DeMarchi2005,SH16b,Wenzel2020,Wittwar2019x,Santin2021}.

\subsection{A kernel-based center-manifold theorem}\label{sec:kernel_based_cm}
Our goal is now to derive error bounds for the approximation process of Theorem \ref{th:representer}. We aim to use these results in combination with Theorem \ref{th:data_based_cm1} and Theorem \ref{th:perturbation}, so we need bounds which contain terms that depend on $\|x\|^q$, with a suitable power $q$.

Comprehensive approximation results are available if plain kernel-based interpolation or regularized interpolation are used instead of the approximant of Theorem \ref{th:representer}, i.e., if no constraint is enforced. For scalar-valued target functions we refer to \cite{Wendland2005} for approximation results on interpolation and to \cite{WendlandRieger2005} to account additionally for a constant regularization term. The extension to vector-valued outputs and general regularization functions has been very recently presented in \cite{Wittwar2020thesis}.

In general terms, to obtain convergence we first assume that the sets $X_{\hat N}$ are increasingly fine. A way to measure this is to require that the fill 
distance $\theta_{\hat N}$ of $X_{\hat N}$ in $\Omega$, defined as
\begin{align*}
\theta_{\hat N}:=\sup\limits_{x\in\Omega}\min\limits_{y\in X_{\hat N}}\|x-y\|_2,
\end{align*}
is decreasing as ${\hat N}$ increases. Observe that in our case it is sufficient that $\Omega$ is a neighborhood of zero and not the full 
space where the ODE is defined. 

Under the assumption that $h$ is in the Hilbert space $\mathcal H$, it holds
\begin{align}\label{eq:base_error_bound}
\sup\limits_{x\in\Omega}\|h(x)-\hat h(x)\|\leq c' F(\theta_{\hat N}) \|h\|_{\mathcal H}=: c F(\theta_{\hat N}), 
\end{align} 
where $F(\theta_{\hat N})$ is a value that converges to zero as $\theta_{\hat N}$ converges to zero, with a rate that depends on the particular kernel. We remark that both 
exponential and algebraic rates of convergence can be achieved (see e.g. \cite{Wendland2005} for a detailed presentation of this kind of error bounds in the scalar-valued case, and \cite{Wittwar2018} for an extension to the vector-valued case). 

As a special case, one can consider sets of points such that 
\begin{align}\label{eq:fill_dist_decay}
\theta_{\hat N} \leq c_d {\hat N} ^{-1/d},    
\end{align}
with a dimension-dependent constant $c_d$. In this case, the bound on the right hand-side of \eqref{eq:base_error_bound} becomes a function of ${\hat N}$ only. We 
remark that the $P$-greedy algorithm of Section \ref{sec:greedy} selects sets of points that have this property at a very convenient computational cost (see 
\cite{SH16b}).

The condition $h\in\mathcal H$ may be difficult to verify or to achieve in general. Nevertheless for some kernels one gets the notable situation where $\mathcal H$ is norm equivalent to the Sobolev space $H^t(\Omega;\R^m)$ with some $t>d/2$. We refer to \cite{Wendland2005} for a discussion of this property for $m=1$, and to \cite{Wittwar2020thesis} in the general case.

In this case the convergence results require only the smoothness 
assumptions $h\in H^t(\Omega;\R^m)$, which seems reasonable in this context, and the error bound \eqref{eq:base_error_bound} reads
\begin{align*}
\sup\limits_{x\in\Omega}\|h(x)-\hat h(x)\|\leq c \theta_{\hat N} ^{t-d/2}.
\end{align*}

This result is clearly not suitable to be used in combination with Theorem \ref{th:data_based_cm1} and Theorem \ref{th:perturbation}, since the upper bounds on the error are uniform in $\Omega$, while the cited theorems require the error to decrease towards $x=0$ with a given speed. The situation is different with the approximant defined in Theorem \ref{th:representer} and in this case we can prove the following theorem, which actually provides suitable bounds with $q=2$. 

\begin{theorem}\label{th:error}
 Let $X_{\hat N} \subset \Omega \setminus \{0\}$ be a set of pairwise distinct points and let $\theta_{\hat N}$ be its corresponding fill distance. Assume that $K$ is such that $\mathcal H$ is norm equivalent to the Sobolev space $H^t(\Omega;\R^m)$ with $t > d/2 + 2$.
 Then, if $h \in H^{t}(\Omega;\R^m)$, and if $\hat{h}$ is the approximant of Section 3, it holds
 \begin{align}\label{eq:approx_error_bound}
     \|h(x) - \hat{h}(x)\| \leq C \theta_{\hat N}^{-2} \left(\theta_{\hat N}^{t-d/2} + r_{\max} \right)\|h\|_{H^t} \|x\|^2 \;\;\forall x\in\Omega,
 \end{align}
where $C>0$ is a constant independent of $\hat N$, $h$ and $x$, and where 
\begin{align*}
r_{\max}:= \max_{\xi \in B_{\|x\|}(0)}\left\{\sqrt{\eta}: \eta \text{ is an eigenvalue of } r(\xi)\in\R^{m\times m}\right\},
\end{align*}
i.e., $r_{\max} = \lambda$ if the constant regularization function \eqref{rem:constant_regularization} is used.

In particular, if $r$ is chosen so that $r_{\max} \leq \theta_{\hat N} ^{t-d/2}$, then we have the bound
 \begin{align*}
     \|h(x) - \hat{h}(x)\| \leq 2 C \theta_{\hat N}^{t-d/2-2} \|h\|_{H^t} \|x\|^2 \;\;\forall x\in\Omega.
 \end{align*}
\end{theorem}
\begin{proof}
 By assumption on $\hat{h}$ we have
 \begin{align*}
     \hat{h}(0) = h(0) = 0, \quad D\hat h(0) = D h(0) =0,
 \end{align*}
 and therefore, by using a Taylor expansion with suitable remainder, there exists $\xi \in B_{\|x\|}(0):=\{ x^{\prime} \in \Omega \, : \, \| x^{\prime} \| \leq \| x \| \}$ such that 
 \begin{align}\label{eq:intermediate}
     \| h(x) - \hat{h}(x)\| \leq \max\limits_{a\in \N^d, |a| = 2 }\;\;\left(\max\limits_{\xi\in B_{\|x\|}(0)} \|\partial^{a} h(\xi) - \partial^{a} \hat{h}(\xi) \|\right) \|x\|^2.
 \end{align}
Now, Theorem 3.2.6 in \cite{Wittwar2020thesis} for $|a|=2$ provides the bound
\begin{align*}
\max\limits_{\xi\in B_{\|x\|}(0)} \|\partial^{a} h(\xi) - \partial^{a} \hat{h}(\xi) \|
\leq C \theta_{\hat N}^{-2} \left(\theta_{\hat N}^{t-d/2} +r_{\max} \right)\|h\|_{H^s},
 \end{align*}
and inserting this result in \eqref{eq:intermediate} proves \eqref{eq:approx_error_bound}. 

The final bound is obtained by substituting the upper bound on $r_{\max}$.
\end{proof}

\begin{remark}[Higher values of $q$]
As discussed in Remark \ref{rem:higher_order}, it is possible to construct approximants with higher-order derivative constraints. In this case, if all the derivatives $\partial^a$ with $|a|\leq p$ are used, it is possible to adapt the proof of Theorem \ref{th:error} to obtain an error bound like
\begin{align*}
     \|h(x) - \hat{h}(x)\| \leq c \theta_{\hat N}^{t-1-p-d/2}\|x\|^{1+p}.
\end{align*}
We avoid reporting the full proof of this fact since it requires no technical modification over the proof of Theorem \ref{th:error}, but rather just the use of a more involved index set. 
\end{remark}

\begin{remark}[Error bounds in terms of $\hat N$]\label{rem:error_vs_n}
If the sequence of sets of points $X_{\hat N}$ satisfies the condition \eqref{eq:fill_dist_decay}, then also in the case of Theorem \ref{th:error} the bound can be rewritten in terms of $\hat N$, and it reads
 \begin{align*}
     \|h(x) - \hat{h}(x)\| \leq c' {\hat N}^{-t/d+2/d+1/2}\|x\|^2,
 \end{align*}
 for a suitable constant $c'$. We recall that the condition \eqref{eq:fill_dist_decay} is satisfied by the points selected by the $P$-greedy algorithm (see \cite{Wenzel2020}).
\end{remark}

\section{Numerical Examples} \label{sec: Numerical examples}
We test now our method on three different examples. 
In each example we start by identifying the splitting in \eqref{eq:full_order_system} manually, and then by solving numerically the IVPs obtained by coupling the ODE system \eqref{eq:full_order_system} with the $2^n$ different initial values $(x_0, y_0)\in\set{\pm 0.8}^n$, where $n$ is the dimension of the full system in each example. As numerical solver we use in all cases the implicit Euler scheme for initial time $t_0 = 0$, final time $T = 1000$ and with the time step $\Delta t = 0.1$. The points $(x_i, y_i)$ lying on the computed numerical trajectories are first filtered to keep only those values that are sufficiently close to the equilibrium $0\in\R^n$, and this is done by removing all points outside of the cube $[-0.1, 0.1]^n$. These first steps generate a dataset of $N$ pairs $(x_i, y_i)$, on which the greedy algorithm of Section \ref{sec:greedy} is run to possibly further reduce the set of points to a final number of $\hat N$ pairs. The tolerance $\varepsilon_{tol}$ which determines the stopping of the greedy algorithm (see Section \ref{sec:greedy}) is set to values that are specified in each example.
The resulting datasets $(X_{\hat N}, Y_{\hat N})$ are finally used to construct the approximants by solving the linear system \eqref{eq:thelinearsystem}, using in each case a constant regularization function with suitable values of $\lambda$  (see Remark \ref{rem:constant_regularization}).
Moreover, all the following center manifolds are scalar-valued, and thus plain positive definite kernels (i.e., matrix valued kernels with $m=1$) are used.
Further details on the setting and the parameters used to build and visualize our approximation to the center manifold are specified in each single example.

As a comparison, we also compute an algebraic approximation of the center manifold with the toolbox \cite{roberts_software}, which provides an approximated expansion of $h$ on the $d$-dimensional monomial basis. The toolbox allows to compute approximants up to a given monomial degree, and in the following we always use the maximal possible value, which is $5$.  The toolbox requires as inputs the right hand side of the ODE, the value of the eigenvalues with zero real part, and a basis of the corresponding eigenspaces both for the linearization matrix and for its transpose.

To compute errors, we use an equally spaced test grid $X_{test}$ made of $1001$ points in $[-0.1, 0.1]$ for the first two examples, and of $101^2$ points in $[-0.1, 0.1]^2$ in the third one. 
Additionally, using the PDE \eqref{cm_pde} we define the pointwise residual
\begin{equation}\label{eq:res_error}
 r(x) := D_x\hat h(x)\left(L_1x +N_1(x,\hat h(x))\right) - \left(L_2\hat h(x) +  N_2(x,\hat h(x))\right),
\end{equation}
and we use it as a further way to quantify how well the approximated manifold models the true one.

\begin{remark}[Polynomial kernel]
Since an algebraic approximation of the manifold is sufficient in many cases, the polynomial kernel $k(x, y):= (1 + \gamma x^T y)^p$, $\gamma>0$, $p>0$, seems particularly promising for the task. This kernel is positive definite, but not strictly positive definite. Moreover, the approximant \eqref{eq: representer thm} in this case takes the particularly simple form
\begin{align*}
\hat h(x) = \sum_{i=1}^{\hat N} k(x, x_i) \alpha_i + \alpha_{\hat N+1} + p \gamma \beta^T x, 
\end{align*}
i.e., the terms corresponding to the interpolation conditions in zero are adding a linear term to the approximant.

\end{remark}

\subsection{Example 1}\label{sec:example1} 
We consider the $2$-dimensional system 
\begin{equation} \label{eq: Example 1 system}
 \begin{split}
  \dot{x} & = L_1x + N_1(x,y) = 0 + xy \\
  \dot{y} & = L_2y + N_2(x,y) = -y - x^2.  
 \end{split}
\end{equation}
Observe that in this case an explicit form of the true center manifold is not known. 

Using the algebraic approximation toolbox \cite{roberts_software}, we get the approximation 
\begin{align}\label{eq:h_alg_1}
\hat h_{alg}(x)=-x^{2}-2  x^{4}+\mathcal O (x^5).
\end{align}

We then generate the dataset for the construction of the kernel approximants, which contains $N = 38248$ data pairs. Starting from these we run the greedy algorithm with $\varepsilon_{tol} = 10^{-15}$ and with several setups: We use the polynomial kernels 
\begin{align*}
    k_{pol}^p(x,y) := (1 + xy/2)^p, \quad p = 4,5,6,
\end{align*}
and the Gaussian kernels 
\begin{align*}
    k_{Gauss}^{\varepsilon}(x,y) := e^{ - \varepsilon \| x-y\|^2}, \quad \varepsilon = 1,5,
\end{align*}
which result in the sets $X_{pol}^4,X_{pol}^5,X_{pol}^6,X_{Gauss}^1$ and $X_{Gauss}^5$, containing 
$14,15,15,7$ and $8$ points, respectively. The sets of target values are selected from $Y_N$ accordingly.

For each dataset we compute a corresponding approximation and residual, denoted as
\begin{align*}
\hat h_{poly}^p(x), r_{pol}^p(x) \quad p = 4,5,6, \quad \text{ and } \quad \hat h_{Gauss}^{\varepsilon}(x),r_{Gauss}^{\varepsilon}(x) \quad \varepsilon = 1,5,
\end{align*}
using the regularization parameter $\lambda := 10^{-10}$ for the Gaussian kernels and $\lambda :=10^{-13}$ for the polynomial kernels. 

As an example, for the kernels $k_{pol}^4$ and $k_{Gauss}^1$ the approximant and the corresponding residual are depicted in Figure~\ref{fig:example1}. In the 
figure we show also the approximation $\hat h_{alg}$ and the residual $r_{alg}(x)= 12 x^6+16 x^8$, which is obtained by manually computing the value in 
\eqref{eq:res_error} using the first two terms in \eqref{eq:h_alg_1}. It is clear from the left panel in Figure \ref{fig:example1} that the three approximants 
are very close. Moreover, the magnitude of the residual (right panel in Figure \ref{fig:example1}) is in all the three cases of the order of $10^{-5}$, with the 
algebraic approximation giving the smallest one. 

\begin{figure}[ht!]
\begin{center}
\begin{subfigure}{.45\textwidth}
\begin{tikzpicture}
    \begin{axis}[legend pos=outer north east,xlabel={x},ylabel={},title={},height=6cm,grid=both,grid style=dashed,xmin=-0.1,xmax=0.1, scaled x ticks=false, x 
tick label style = {/pgf/number format/.cd,fixed,precision=5}  ]
      \addplot[mark=none,color=red,line width=1.5] table[col sep=comma, x=input, y=manifold1] {example1_with_algebraic.csv};
      \addlegendentry{$\hat h_{poly}^4$};
      \addplot[mark=none,color=blue,dashed,line width=1.5] table[col sep=comma, x=input, y=manifold2] {example1_with_algebraic.csv};
      \addlegendentry{$\hat h_{Gauss}^1$};
      \addplot[mark=none,color=green,dotted,line width=1.5] table[col sep=comma, x=input, y=manifold3] {example1_with_algebraic.csv};
      \addlegendentry{$\hat h_{alg}$};
      \end{axis}
  \end{tikzpicture}
\end{subfigure}
\hfill
\begin{subfigure}{.45\textwidth}
\begin{tikzpicture}
    \begin{axis}[legend pos=outer north east,xlabel={x},ylabel={},title={},height=6cm,grid=both,grid style=dashed, scaled x ticks=false, x tick label style = 
{/pgf/number format/.cd,fixed,precision=5},xmin=-0.1,xmax=0.1]
      \addplot[mark=none,color=red,line width=1.5] table[col sep=comma, x=input, y=residual1] {example1_with_algebraic.csv};
      \addlegendentry{$r_{pol}^4$};
      \addplot[mark=none,color=blue,dashed,line width=1.5] table[col sep=comma, x=input, y=residual2] {example1_with_algebraic.csv};
      \addlegendentry{$r_{Gauss}^1$};
      \addplot[mark=none,color=green,dotted,line width=1.5] table[col sep=comma, x=input, y=residual3] {example1_with_algebraic.csv};
      \addlegendentry{$r_{alg}$};
      \end{axis}
  \end{tikzpicture}
\end{subfigure}
\end{center}
\caption{Approximations $\hat h_{poly}^4$ and $\hat h_{Gauss}^1$ of the center manifold (left) and corresponding residuals $r_{pol}^4$ and $r_{Gauss}^1$ (right) for the example of Section \ref{sec:example1}.}\label{fig:example1}
\end{figure}
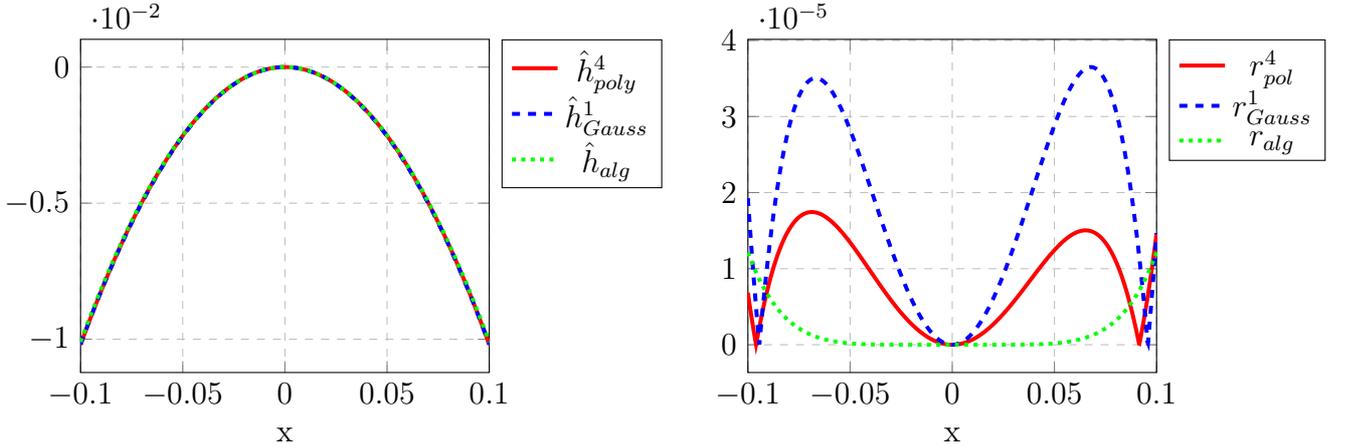

For each approximant, we first check how close $\hat h$ is to meet the constraints of the optimization problem of Theorem \ref{th:representer}. The first two columns of Table \ref{tab:example1_values} show the values $\hat h(0)$ and $D \hat h(0)$, which should be zero and indeed are in all cases at least smaller than $10^{-9}$ in magnitude, and exactly zero for $k^5_{poly}$ and $k^5_{gauss}$. We argue that the non-zero values are due to numerical inaccuracies when solving the system \eqref{eq:thelinearsystem}, and they should be considered numerical zeros. Indeed, the interpolation error $\max_{x_i\in X_{\hat N}}  |\hat h(x_i) - y_i|$ on the training set, which is shown in the third column of the table, is several order of magnitudes larger than these non-zero values.

We further verify if the approximated manifolds can be used to derive information on the asymptotic stability of the reduced system. The reduced system in this case reads
\begin{align*}
\dot x = x \hat h(x),
\end{align*}
and thus $x=0$ is an equilibrium independently of the values of $\hat h(x)$. Since the algebraic approximation \eqref{eq:h_alg_1} is correct up to an error of $\mathcal O(x^5)$, we may assume that this equilibrium is asymptotically stable, since the right hand side is negative on the left of $x=0$ and positive on the right. Under this assumption, we may conclude that an approximant $\hat h$ provides the correct characterization of the stability of the equilibrium of the reduced system if it is strictly negative in a neighborhood of zero, and zero only in zero. Looking at the fourth and first columns of Table \ref{tab:example1_values}, it is the case that $\max_{x\in X_{test}\setminus\{0\}} \hat h(x)< 0$ and $\hat h (0) \leq 0$ in all cases except for $\hat h^6_{poly}$, and thus these four approximants are sufficiently accurate to predict the stability of the equilibrium. In the case of $\hat h^6_{poly}$ we have still that $\max_{x\in X_{test}\setminus\{0\}} \hat h(x)< 0$, but the approximant is positive in zero. Since it is continuous, a finer test grid would find non negative values also for $x\neq 0$, and thus no stability is found in this case.

\begin{table}[ht!]
\begin{center}
\begin{tabular}{|l|c|c|c||c|}
\hline
& $\hat h(0)$ & $D \hat h(0)$ & $\max_{x_i\in X_{\hat N}}  |\hat h(x_i) - y_i|$ & $\max_{x\in X_{test}\setminus\{0\}} \hat h(x)$\\
\hline
$\hat h_{poly}^4$   &-3.73e-09 & -1.04e-10 & 4.72e-07&  -3.88e-08\\
$\hat h_{poly}^5$   &0 & -8.00e-10 & 5.95e-07&  -2.88e-08\\
$\hat h_{poly}^6$   &\phantom{-}3.73e-09 & \phantom{-}5.82e-10 & 5.89e-07&  -2.78e-08\\
$\hat h_{Gauss}^1$ &-1.46e-11 & -1.36e-12 & 5.17e-06& -4.01e-08\\
$\hat h_{Gauss}^5$ &0 & -6.82e-12 & 1.52e-06&  -3.99e-08\\
\hline
\end{tabular}
\end{center}
\caption{Values of the approximated center manifolds and of their derivatives in zero and on the test set, for various choices of the kernel and its parameter.}\label{tab:example1_values}
\end{table}

Finally, we expand the approximants as Taylor series of order $4$ to draw a further comparison to the results given by \cite{roberts_software}. These expansions are obtained by computing explicitly the expansions of the kernel terms in \eqref{eq: representer thm} and then numerically computing the corresponding linear combinations. The coefficients of the Taylor expansions are given in Table \ref{tab:example1_taylor}. The accuracy also in this case is very good, even if it is degrading as the degree of the monomials increases. For the first three terms the best approximation is provided by the approximants using the Gaussian kernel.

\begin{table}[ht!]
\begin{center}
\begin{tabular}{|l|c|c|c|c|c|}
\hline
&\multicolumn{5}{c|}{Monomial}\\
\hline
& $1$ & $x$ & $x^2$ & $x^3$ & $x^4$ \\
\hline
$\hat h_{poly}^4$   &-3.73e-09           & -1.04e-10 & -9.99e-01&  \phantom{-}1.95e-04& -2.19e+00\\
$\hat h_{poly}^5$   &0                   & -5.82e-10 & -9.99e-01&  -1.97e-04          & -2.20e+00\\
$\hat h_{poly}^6$   &\phantom{-}3.73e-09 & -1.86e-09 & -9.99e-01&  -1.73e-05          & -2.20e+00\\
$\hat h_{Gauss}^1$ &\phantom{-}1.46e-11 & -1.36e-12 & -1.00e+00&  -6.57e-04          & -1.86e+00\\
$\hat h_{Gauss}^5$ &0                   & -7.73e-12 & -9.98e-01&  \phantom{-}1.36e-03& -2.52e+00\\
\hline
$\hat h_{alg}$     &0         &        0  &        -1&          0&        -2\\
\hline
\end{tabular}
\end{center}
\caption{Coefficients of the Taylor expansion in $x=0$ of the approximated center manifolds for various choices of the kernel and its parameter, and for the algebraic approximant $\hat h_{alg}$.}\label{tab:example1_taylor}
\end{table}

\subsection{Example 2}

We consider the $2$-dimensional system
\begin{equation} \label{eq: Example 2 system}
 \begin{split}
  \dot{x} & = L_1x + N_1(x,y) = 0 -xy \\
  \dot{y} & = L_2y + N_2(x,y) = -y + x^2 - 2y^2.  
 \end{split}
\end{equation}
This example is of relevance for this paper especially because an exact center manifold $h(x) = x^2$ is explicitly known (see \cite{Roberts1985}, and \cite{Roberts2019} for a recent analysis of this derivation and a study of other properties of the system). Using the toolbox in \cite{roberts_software}, we get that $h(x)=x^{2}+\mathcal O (x^5)$.

For the kernel approximation this time we use the compactly supported Wendland kernel for space dimension $p_d=1$ and smoothness $2 p_k$, with $p_k=1$ (see \cite{Wendland1995a}) given by
\begin{align*}
    k_{wendland}(x,y) := (1 - |x-y|)_{+}^3(1+3|x-y|) := \max(0, 1 - |x-y|)^3(1+3|x-y|),
\end{align*}
whose RKHS is the Sobolev space $H^{p_d/2 + 1/2 + p_k}(\Omega) = H^2(\Omega)$. Therefore, we have $h \in \cH_{k_{wendland}}$ and hence we are able to apply the error theory of Section 
\ref{sec:kernel_based_cm} and Theorem \ref{th:error}. 

For this purpose we again run the $P$-Greedy algorithm with a tolerance of $\varepsilon_{tol} = 10^{-10}$ which results in a point set $X_{wendland} = 
\{x_1,\dots,x_{200}\}$ containing $\hat N = 200$ elements. Observe that these points are selected with an ordering corresponding to a sequentially improving 
covering of the full dataset (see \cite{DeMarchi2005}). 
To test the behavior of the error, we construct the full sequence of approximants $\hat h_{\hat N}$ corresponding to each of the incremental point sets $X_{\hat N} := 
\{x_1,\dots,x_{\hat N}\}$, $1\leq {\hat N}\leq 200$. We use the regularization value $\lambda:= 10^{-13}$, and for each approximant we compute the error 
\begin{align*}
    e_{\hat N} := \max\limits_{x \in X_{test}} |h(x) - \hat h_{\hat N}(x)|.
\end{align*}
This discrete error serves as a proxy of the maximum error in \eqref{eq:approx_error_bound}. The values of $e_{\hat N}$ are shown in the second row in Figure \ref{fig:example2}, and they clearly have a quadratic rate of decay to zero as a function of ${\hat N}$. Observe that this rate of convergence is significantly faster than the one predicted by Theorem \ref{th:error} and Remark \ref{rem:error_vs_n}, which is equal to $-s/d +1/d + 1/2 = -2 + 3/2 = -1/2$.

The pointwise error for the final approximant $\hat h_{200}$ is depicted instead in the first row in Figure \ref{fig:example2}, left, and it is below $2 \cdot 10^7$ in the entire interval. A plot of the error restricted to the smaller interval $[-0.01, 0.01]$, and computed by scaling $X_{test}$ to this interval, is shown in the right part of the the first row, and it is compared with the decay to zero of the functions $|x|^3$ and $|x|^4$. On a qualitative level, this figure seems to suggests that the error around zero (see Theorem \ref{th:data_based_cm1} and Theorem \ref{th:error}) is of order $q=3$, but not $q=4$.

\begin{figure}[ht]
\begin{subfigure}{.45\textwidth}  

\begin{tikzpicture}
    \begin{axis}[legend pos= north east,xlabel={x},ylabel={},title={},height=6cm,grid=both,grid style=dashed, scaled x ticks=false, x tick label style = 
{/pgf/number format/.cd,fixed,precision=5},xmin=-0.1,xmax=0.1, ymax=0.0000002]
      \addplot[mark=none,color=red,line width=1.5] table[col sep=comma, x=input, y=error]{example2_wendland_error.csv};
      \addlegendentry{$e_{200}$};
      \end{axis}
\end{tikzpicture}

\end{subfigure}
\hfill
\begin{subfigure}{.45\textwidth}  

\begin{tikzpicture}
    \begin{axis}[legend pos= north east,xlabel={x},ylabel={},title={},height=6cm,grid=both,grid style=dashed, scaled x ticks=false, x tick label style = 
{/pgf/number format/.cd,fixed,precision=5},xmin=-0.01,xmax=0.01,restrict y to domain=0:0.000000001]
      \addplot[mark=none,color=red,line width=1.5] table[col sep=comma, x=input, y=error]{example2_wendland_error_zoom.csv};
      \addlegendentry{$e_{200}$};
      \addplot[mark=none,dashed,color=blue,line width=1.5] table[col sep=comma, x=input, y=power_3]{example2_wendland_error_zoom.csv};
      \addlegendentry{$|x|^3$};
      \addplot[mark=none,dotted,color=green,line width=1.5] table[col sep=comma, x=input, y=power_4]{example2_wendland_error_zoom.csv};
      \addlegendentry{$|x|^4$};
      \end{axis}
\end{tikzpicture}

\end{subfigure}
\\
\begin{center}
\begin{subfigure}{.45\textwidth}  

\begin{tikzpicture}
    \begin{axis}[legend pos= north east,xlabel={$\hat N$},ylabel={},title={},height=6cm,grid=both,grid style=dashed, scaled x ticks=false, x tick label style = 
{/pgf/number format/.cd,fixed,precision=5},xmax=200,xmode=log,ymode=log]
      \addplot[mark=none,color=red,line width=1.5] table[col sep=comma, x=N, y=e_N]{example2_wendland_decay.csv};
      \addlegendentry{$e_{\hat N}$};
      \addplot[mark=none,color=blue,dashed,line width=1.5] table[col sep=comma, x=N, y=rate]{example2_wendland_decay.csv};
      \addlegendentry{$10 ^ {-2} {\hat N}^{-2}$};
      \end{axis}
\end{tikzpicture}

\end{subfigure}
\end{center}

\caption{Pointwise error over the domain $\Omega$ for the last approximant of the sequence, i.e., $\hat h_{200}$ (shown on two different intervals, first row), and decay of the maximum error $e_j$ with respect to the number of data points used in the construction of the approximant, compared with the rate $10^{-2} \hat N^{-2}$ (second row).}
\label{fig:example2}
\end{figure}
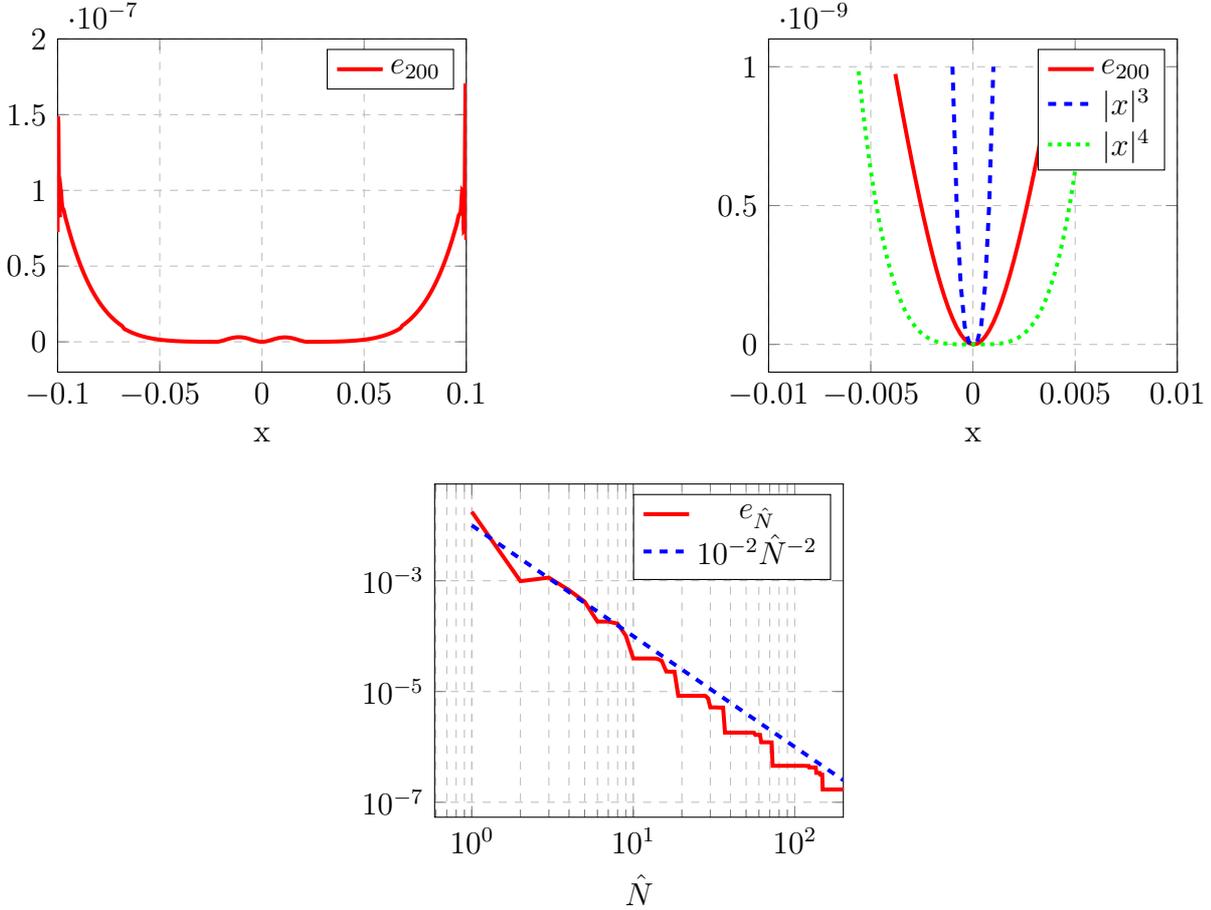

\subsection{Example 3}
We consider now the $(2+1)$-dimensional system
\begin{equation} \label{eq: Example 3 system}
 \begin{split}
  \dot{x} & = L_1 x + N_1(x,y) =
              \begin{pmatrix}
               0 & -1 \\ 1 & \phantom{-}0 
              \end{pmatrix}
              \begin{pmatrix}
               x_1 \\ x_2
              \end{pmatrix}
              + y
              \begin{pmatrix}
               x_1 \\ x_2 
              \end{pmatrix} \\
    \dot{y} & = L_2y + N_2(x,y) = -y -x_1^2 - x_2^2 + y^2.
 \end{split}
\end{equation}
In this case the toolbox \cite{roberts_software} returns an implicit parametrization of the manifold depending on the variables $t, s_1, s_2$, and given by the equations
\begin{align}\label{eq:alg_man_ex3}
\nonumber x_1 &= \frac{\sqrt{2}}{2}\left(e^{it} s_1 + e^{-it} s_2\right)\\
x_2 &= \frac{\sqrt{2}}{2} i \left(-e^{it}  s_1 + e^{-it}  s_2\right)\\
\nonumber y &= -2 s_1 s_2 -12 s_2^2 s_1^2.
\end{align}
This system requires some manipulation to be in the form that we are interested in, since it is not directly providing an expression of $y$ in terms of $x_1, x_2$. Observe that this is a limitation to the use of this methodology as a model order reduction technique.

To derive the expression for the manifold we invert the first two equations in \eqref{eq:alg_man_ex3}, which are linear and invertible, to obtain
\begin{align*}
s_1 &= \frac{\sqrt{2}}{2}\left(e^{-it} x_1 + i e^{-it} x_2\right)\\
s_2 &= \frac{\sqrt{2}}{2}\left( e^{it} x_1 - i e^{it} x_2\right),
\end{align*}
and these imply that $s_1 s_2 = \frac{1}{2}\left(e^{-it} x_1 + i e^{-it} x_2\right)\left( e^{it} x_1 - i e^{it} x_2\right) = \frac{1}{2}\left(x_1^2 + x_2^2\right)$. Inserting this relation into the third equation in \eqref{eq:alg_man_ex3} gives finally 
\begin{equation*}
y =  -2 s_1 s_2 -12 s_2^2 s_1^2 = - \left(x_1^2 + x_2^2\right) - 3 \left(x_1^2 + x_2^2\right)^2=:\hat h_{alg}(x_1, x_2),
\end{equation*}
where also in this case the approximation is limited to terms of degree at most $4$.

The initial dataset to compute the kernel approximants in this case consists of $N = 78796$ data pairs. We use the kernels $k_{poly}^4(x,y) := (1 + x^Ty/2)^4$ and 
$k^{1/2}_{Gauss}(x,y) := e^{- \|x - y\|_2^2/2}$, and the greedy-selected sets with $\varepsilon_{tol}:=10^{-10}$ have size $\hat N = 10$ in both cases.

The approximations $\hat h_{poly}^4$, $\hat h_{Gauss}^{1/2}$, which use both $\lambda:=10^{-10}$, and their corresponding residuals $r_{poly}^4$, $r_{Gauss}^{1/2}$ can be seen in Figure \ref{fig: example 3}. The plots show a good accuracy for both kernels and among the two the polynomial kernel provides the best results. Observe that the error is small especially in a neighborhood of zero, as expected.

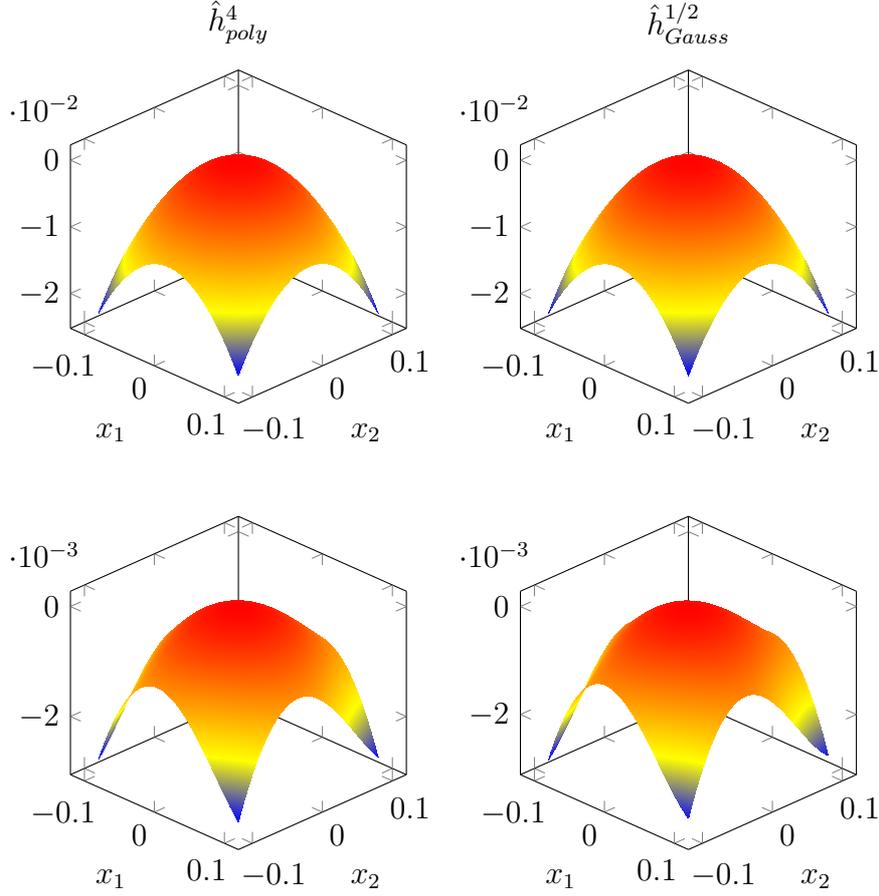
\begin{figure}[ht]
\begin{center}

\begin{tikzpicture}
\begin{groupplot}[
    group style={group size=2 by 2, vertical sep=1.5cm,horizontal sep=1.5cm},xmin=-0.12,xmax=0.12,ymin=-0.12,ymax=0.12, xlabel shift = -0.4cm, ylabel shift = 
-0.4cm
   ]

\nextgroupplot[
       title = {$\hat h_{poly}^4$},
       view/h=45,
       xlabel=$x_1$, 
       ylabel=$x_2$, 
       scaled x ticks=false, 
       scaled y ticks=false, 
       x tick label style = {/pgf/number format/.cd,fixed,precision=5}, 
       y tick label style = {/pgf/number format/.cd,fixed,precision=5},
       height=6cm, 
       width=6cm,
       ]
\addplot3[surf,mesh/ordering=y varies, shader=interp] table{example3_poly_app.csv};

\nextgroupplot[
       title = {$\hat h_{Gauss}^{1/2}$},
       view/h=45,
       xlabel=$x_1$, ylabel=$x_2$,
       scaled x ticks=false, scaled y ticks=false, x tick label style = {/pgf/number format/.cd,fixed,precision=5}, y tick label style = {/pgf/number 
format/.cd,fixed,precision=5}  ,height=6cm, width=6cm     
    ]
\addplot3[surf,mesh/ordering=y varies,
    shader=interp]
    table {example3_gauss_app.csv};
\nextgroupplot[
       view/h=45,
       xlabel=$x_1$, ylabel=$x_2$, 
       scaled x ticks=false, scaled y ticks=false, x tick label style = {/pgf/number format/.cd,fixed,precision=5}, y tick label style = {/pgf/number 
format/.cd,fixed,precision=5},height=6cm, width=6cm       
    ]
\addplot3[surf,mesh/ordering=y varies,
    shader=interp]
    table {example3_poly_res.csv};
\nextgroupplot[
       view/h=45,
       xlabel=$x_1$, ylabel=$x_2$,
       scaled x ticks=false, scaled y ticks=false, x tick label style = {/pgf/number format/.cd,fixed,precision=5}, y tick label style = {/pgf/number 
format/.cd,fixed,precision=5},height=6cm, width=6cm       
    ]
\addplot3[surf,mesh/ordering=y varies,
    shader=interp]
    table {example3_gauss_res.csv};
\end{groupplot}
\end{tikzpicture}

\caption{Approximations $\hat h_{poly}^4$ and $\hat h_{Gauss}^{1/2}$ of the center manifold (first row), and corresponding residuals $r_{poly}^4$ and $r_{Gauss}^{1/2}$ (second row), for the ODE \eqref{eq: Example 3 system}. \label{fig: example 3}}
\end{center}
\end{figure}

Following Theorem \ref{th:perturbation}, we now try to simulate the evolution of the system \eqref{eq: Example 3 system} using the approximated reduced dynamics instead of the exact reduced dynamics. Since we do not know the center manifold explicitly, the trajectories of the exact reduced system are obtained by solving the full system, and then considering only the first two components of each trajectory point.

In more details, we fix an initial condition $(x_1, x_2, y)_0$ (and its coordinate projection $(x_1, x_2)_0$) and we solve numerically the exact system \eqref{eq: Example 3 system} and the approximated system 
\begin{align*}
\dot x  & = \begin{pmatrix}
               0 & -1 \\ 1 & \phantom{-}0 
              \end{pmatrix}
              \begin{pmatrix}
               x_1 \\ x_2
              \end{pmatrix}
              + \hat h^{4}_{poly}(x_1, x_2)
              \begin{pmatrix}
               x_1 \\ x_2 
              \end{pmatrix},
\end{align*}
with the same setting used in all the experiments. For $t_i = t_0, t_0+\Delta t, \dots, T$ we obtain a trajectory $\{(x_1, x_2, y)_{t_i}\}_{t_i=t_0}^T$ of the full system, from which we consider the first two coordinates $\{(x_1, x_2)_{t_i}\}_{t_i=t_0}^T$, and a trajectory $\{(\hat x_1, \hat x_2)_{t_i}\}_{t_i=t_0}^T$ of the approximated and reduced system. 

We consider the two initial conditions $x_0^{repr}:=(x_1, x_2, y)_0 := (-0.8, -0.8, -0.8)$ and  $x_0^{gen}:=(x_1, x_2, y)_0 := (-0.4, 0, -0.8)$. The first one has been used for the training of $\hat h_{poly}^4$ while the second is new, and thus these are intended as a reproduction and a generalization experiment.

The corresponding exact and approximated numerical trajectories are shown in the first row of Figure \ref{fig:example3_trajectories}. It is clear that for $x_0^{repr}$ the evolution of the approximated system remains closer to the exact one than for $x_0^{gen}$, which is the expected behavior. Nevertheless, it is remarkable that in both cases the approximated trajectory converges to the same final state, which is the equilibrium $(0, 0)$, and thus the reduced accuracy of the generalization case is not relevant for the long term evolution of the system. 

This fact may be explainable by looking at the second row of Figure \ref{fig:example3_trajectories}, where we report the absolute and relative distance between the exact and reduced numerical trajectories in the two cases over the entire time interval. Indeed, the absolute error is decaying exponentially for both  $x_0^{repr}$ and $x_0^{gen}$ with a rate which is numerically estimated to be $e^{-0.05 t_i}$. This decay stops at a plateau value at around $t_i=400$. This exponential decay of the error is nevertheless driven by the fact that the norm $\sqrt{x_1^2 + x_2^2}$ of the true numerical solution is exponentially decaying with roughly the same rate, until it reaches a final value (the approximation of the equilibrium) at $t_i\approx400$. Indeed the relative error, for $t_i\leq 350$, is rapidly stabilizing to a constant value of $2.97\%$ (reproduction) and $20.65\%$ (generalization). Thus, a pretty large error in terms of approximation is in fact small in terms of solutions of the IVP.

This observation suggests once again that our approximation scheme may be used as an effective model order reduction technique, since an approximant computed using a few initial values is able to provide meaningful predictions for unseen parameters. This numerical observation would deserve further theoretical investigation.

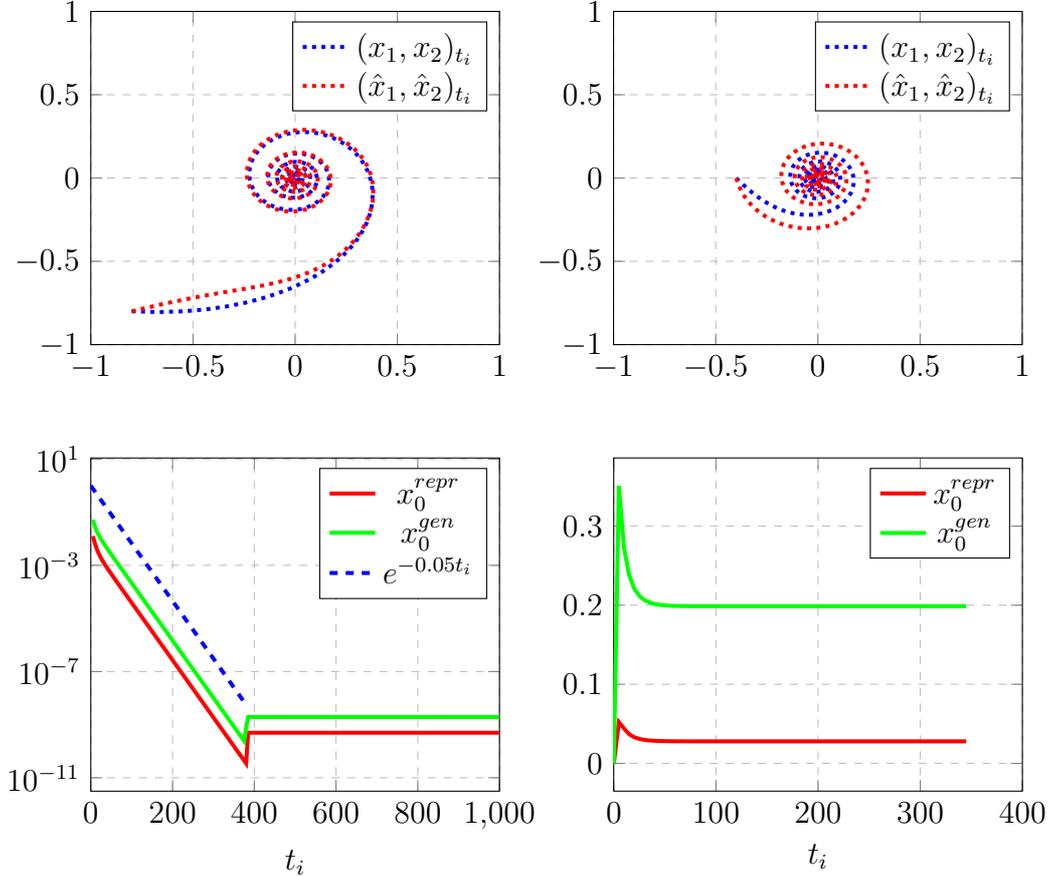
\begin{figure}[ht]
\begin{center}

\begin{tikzpicture}
\begin{groupplot}[
    group style={group size=2 by 2, vertical sep=1.5cm,horizontal sep=1.5cm}]

\nextgroupplot[legend pos= north east,title={},height=6cm,grid=both,grid style=dashed, scaled x ticks=false, x tick label style = 
{/pgf/number format/.cd,fixed,precision=5},xmin=-1,xmax=1, ymin=-1,ymax=1]
      \addplot[mark=none,color=blue,dotted,line width=1.5] table[col sep=comma, x=x_full, y=y_full]{example3_trajectory_reproduction.csv};
      \addlegendentry{$(x_1, x_2)_{t_i}$};
      \addplot[mark=none,color=red,dotted,line width=1.5] table[col sep=comma, x=x_app, y=y_app]{example3_trajectory_reproduction.csv};
      \addlegendentry{$(\hat x_1, \hat x_2)_{t_i}$};
 
\nextgroupplot[legend pos= north east,title={},height=6cm,grid=both,grid style=dashed, scaled x ticks=false, x tick label style = 
{/pgf/number format/.cd,fixed,precision=5},xmin=-1,xmax=1, ymin=-1,ymax=1]
      \addplot[mark=none,color=blue,dotted,line width=1.5] table[col sep=comma, x=x_full, y=y_full]{example3_trajectory_generalization.csv};
      \addlegendentry{$(x_1, x_2)_{t_i}$};
      \addplot[mark=none,color=red,dotted,line width=1.5] table[col sep=comma, x=x_app, y=y_app]{example3_trajectory_generalization.csv};
      \addlegendentry{$(\hat x_1, \hat x_2)_{t_i}$};
 
\nextgroupplot[legend pos= north east,xlabel={$t_i$},ylabel={},title={},height=6cm,grid=both,grid style=dashed, scaled x ticks=false, x tick label style = 
{/pgf/number format/.cd,fixed,precision=5},xmin=-0.1,xmax=1000,ymode=log]
      \addplot[mark=none,color=red,line width=1.5] table[col sep=comma, x=t, y=err_reproduction]{example3_trajectory_error_subsample.csv};
      \addlegendentry{$x_0^{repr}$};
      \addplot[mark=none,color=green,line width=1.5] table[col sep=comma, x=t, y=err_generalization]{example3_trajectory_error_subsample.csv};
      \addlegendentry{$x_0^{gen}$};
      \addplot[mark=none,color=blue,dashed,line width=1.5] table[col sep=comma, x=t_rate, y=rate]{example3_trajectory_rate_subsample.csv};
      \addlegendentry{$e^{-0.05 t_i}$};

\nextgroupplot[legend pos= north east,xlabel={$t_i$},ylabel={},title={},height=6cm,grid=both,grid style=dashed, scaled x ticks=false, x tick label style = 
{/pgf/number format/.cd,fixed,precision=5},xmin=-0.1,xmax=400]
      \addplot[mark=none,color=red,line width=1.5] table[col sep=comma, x=t, y=err_reproduction]{example3_trajectory_rel_error_subsample.csv};
      \addlegendentry{$x_0^{repr}$};
      \addplot[mark=none,color=green,line width=1.5] table[col sep=comma, x=t, y=err_generalization]{example3_trajectory_rel_error_subsample.csv};
      \addlegendentry{$x_0^{gen}$};
\end{groupplot}
\end{tikzpicture}
\end{center}
\caption{The first row shows numerical trajectories computed with the full system (blue curves) and with the reduced and approximated one (red curves) for an initial value $x_0^{repr}$ (left) and for an initial value $x_0^{gen}$ (right). The second row shows the evolution over time of the absolute error compared with the rate $e^{-0.05 t_i}$ (left, with logarithmic scaled) and of the relative error (right, with linear scale and restricted to $t_i\leq 350$) between the full and reduced trajectories.}
\label{fig:example3_trajectories} 
\end{figure}

As a final analysis, we compute the Taylor expansion in $(0, 0)$ of $h_{poly}^4$. Table \ref{tab:example3_taylor} reports the coefficients of the expansion for the monomials up to degree four, computed as in Section \ref{sec:example1}. The terms of degree up to $2$ have good agreement with $\hat h_{alg}$ (which is exact up to degree four). The terms of degree three should be zero, and they are of magnitude $10^{-3}$, while the accuracy is degrading for the highest order terms, where no clear pattern nor resemblance with the exact values can be seen.

\begin{table}[ht!]
\begin{center}
\begin{tabular}{|rr|rr|rr|rr|rr|}
\hline
\multicolumn{2}{|c|}{Degree $0$} & \multicolumn{2}{c|}{Degree $1$} & \multicolumn{2}{c|}{Degree $2$} & \multicolumn{2}{c|}{Degree $3$} & \multicolumn{2}{c|}{Degree $4$} \\
\hline
1&1.46e-11&$x_1$&-9.09e-13&$x_1^2$  &-1.13e+00&$x_1^3$    &-5.56e-03&$x_1^4$      &-1.37e-01\\
 &        &$x_2$&-1.55e-11&$x_1 x_2$&-1.33e-09&$x_1^2 x_2$&-1.35e-03&$x_1^3 x_2^1$& 9.18e-01\\
 &        &     &         &$x_2^2$  &-1.13e+00&$x_1 x_2^2$& 6.23e-03&$x_1^2 x_2^2$&-3.96e+00\\
 &        &     &         &         &         &$x_2^2$    &-1.36e-03&$x_1^1 x_2^3$&-9.18e-01\\
 &        &     &         &         &         &           &         &$x_2^4$      &-1.37e-01\\
\hline
\end{tabular}
\end{center}
\caption{Coefficients of the Taylor expansion in $x=0$ of the approximated center manifold $\hat h_{poly}^4$ up to degree four.}\label{tab:example3_taylor}
\end{table}

\section{Conclusions}
In this paper we proved a general data-based version of the main implication of the center manifold theorem, which allows to draw the same conclusions as the original theorem but examining a sufficiently accurate approximation in place of the true manifold. Additionally, such approximation can be used to evolve the system along approximated trajectories, and we proved bounds on their deviation from the exact ones.

Moreover, we revisited an algorithm that we introduced in a previous work for the kernel-based approximation of the center manifold. We added details on its definition and optimality, and especially we proved rigorous bounds on the error between this approximant and the true manifold, and discussed how to use these bounds in combination with the theorems mentioned above.

Extensive numerical experiments suggested that the present method can reach a significant accuracy, and that it has the potential to be used as an effective technique to analyze the stability of dynamical systems.

Several points remain open and would deserve further investigation. Mainly, establishing a connection between the exponent $q$ in Theorem \ref{th:data_based_cm1} and the full order dynamical system would allow a more systematic application of the results of this paper on different problems. Moreover, results on the generalization power of the method for the simulation of IVPs with initial values not used during training are missing.

\vspace{1cm}
\noindent\textbf{Acknowledgements:}
The first and fourth authors would like to thank the German Research Foundation (DFG) for support within the Cluster of Excellence in
Simulation Technology (EXC 310/2) at the University of Stuttgart. The second author   thanks the European Commission for financial support received through 
Marie Curie 
Fellowships. 

\bibliographystyle{abbrv}
\bibliography{biblio}

\end{document}